\documentclass[11pt]{amsart}
\usepackage{amsmath}
\usepackage[english,  activeacute]{babel}
\usepackage[latin1]{inputenc}
\usepackage{amssymb}
\usepackage{amsthm}
\usepackage{graphics,graphicx}
\usepackage{array}
\usepackage{a4wide}
\allowdisplaybreaks
\usepackage{color, url}
\usepackage{float}
\usepackage{multicol}
\usepackage[shortlabels]{enumitem}
\usepackage[hypertexnames=false,colorlinks=true,allcolors=blue]{hyperref}

\theoremstyle{plain}
\newtheorem{theorem}{Theorem}[section]
\newtheorem{proposition}[theorem]{Proposition}

\newtheorem{lemma}[theorem]{Lemma}
\theoremstyle{definition}

\newtheorem{fact}[theorem]{Fact}
\newtheorem{observation}[theorem]{Observation}
\newtheorem{conjecture}[theorem]{Conjecture}

\newcommand{\csim}{\stackrel{m}{\sim}}
\newcommand{\msim}{\stackrel{m}{\sim}}
\newcommand{\wsim}{\stackrel{w}{\sim}}
\newcommand{\sfsim}{\stackrel{sf}{\sim}}
\newcommand{\ftsim}{\stackrel{ft}{\sim}}
\newcommand{\he}{\text{ht}}
\DeclareMathOperator{\Av}{Av}
\DeclareMathOperator{\av}{av}

\begin{document}
\title[Avoiding in multisets]{Avoiding a pair of patterns in multisets and compositions}
\author[V. Jel\'\i nek]{V\'\i t Jel\'\i nek}
\address{Computer Science Institute, Charles University, Prague, Czechia}
\email{jelinek@iuuk.mff.cuni.cz}
\author[T. Mansour]{Toufik Mansour}
\address{Department of Mathematics, University of Haifa,
3498838 Haifa, Israel}
\email{tmansour@univ.haifa.ac.il}
\author[J. L. Ram\'{\i}rez]{Jos\'e L. Ram\'{\i}rez}
\address{Departamento de Matem\'aticas,  Universidad Nacional de Colombia,  Bogot\'a, Colombia}
\email{jlramirezr@unal.edu.co}
\author[M. Shattuck]{Mark Shattuck}
\address{Department of Mathematics, University of Tennessee,
37996 Knoxville, TN, USA}
\email{shattuck@math.utk.edu}
\thanks{V. Jel{\'\i}nek is supported by project 18-19158S of
the Czech Science
Foundation.}

\begin{abstract}
In this paper, we study the Wilf-type equivalence relations among multiset permutations. We identify
all multiset equivalences among pairs of patterns consisting of a
pattern of length three and another pattern of length at most four.  To establish our results, we
make use of a variety of techniques, including Ferrers-equivalence arguments, sorting by
minimal/maximal letters, analysis of active sites and direct bijections.  In several cases, our
arguments may be extended to prove multiset equivalences for infinite families of pattern pairs.
Our results apply equally well to the Wilf-type classification of compositions, and as a consequence, we
obtain a complete description of the Wilf-equivalence classes for pairs of patterns of type (3,3)
and (3,4) on compositions, with the possible exception of two classes of type (3,4).

\end{abstract}

\date{\today}

\maketitle

\section{Introduction}

A \emph{multiset} is an unordered collection of elements which may be repeated. A \emph{multiset of
height $k$} is a multiset $S$ whose elements are positive integers and whose largest element is $k$. A finite
multiset may be represented by $S=1^{a_1}2^{a_2}\cdots k^{a_k}$, where $a_i\geq 0$ is the
\emph{multiplicity} of $i$ in $S$, i.e., the number of copies of $i$ in $S$. We call a multiset of
height $k$ \emph{reduced} if each member of $[k]=\{1,2,\ldots,k\}$ appears at least once in~$S$, or
equivalently, each multiplicity $a_i$ is at least 1. We will assume, unless otherwise noted, that the
multisets we work with are reduced. The \emph{size} of a multiset~$S$, denoted $|S|$, is the sum of
the multiplicities of its elements.

A \emph{multipermutation} of a multiset $S=1^{a_1}2^{a_2}\cdots k^{a_k}$ is an arrangement of the
elements of $S$ into a sequence. We identify such a multipermutation with a word
$\pi=\pi_1\pi_2\dotsb\pi_n$ that has exactly $a_i$ occurrences of each symbol~$i\in[k]$. The
\emph{height} of $\pi$, denoted $\he(\pi)$, is the height of the
underlying multiset, i.e., the maximum of $\pi_1,\dotsc,\pi_n$.

For two
multipermutations $\rho=\rho_1\dotsb\rho_\ell$ and $\pi=\pi_1\dotsb\pi_n$, we say that $\pi$
\emph{contains} $\rho$, if $\pi$ has a subsequence $\pi_{i(1)}\pi_{i(2)}\dotsb\pi_{i(\ell)}$ whose
elements have the same relative order as $\rho$, i.e., $\pi_{i(a)}<\pi_{i(b)}$ if and only if
$\rho_a<\rho_b$ and $\pi_{i(a)}>\pi_{i(b)}$ if and only if
$\rho_a>\rho_b$ for every $a,b\in[\ell]$. If $\pi$ does not contain $\rho$, it \emph{avoids}~$\rho$.
In this context, $\rho$ is usually referred to as a \emph{pattern}.

For a word $\pi=\pi_1\pi_2\dotsb\pi_n$ of height $k$, its \emph{reversal} is the word
$\pi^r=\pi_n\pi_{n-1}\dotsb\pi_1$, and its \emph{complement} is the word
$\pi^c=k+1-\pi_1,k+1-\pi_2,\dotsc, k+1-\pi_n$. Note that the reversal represents the same multiset as
$\pi$, while the complement may represent a different one.

For a multiset $S$ and a multipermutation $\rho$, we let $\Av(S;\rho)$ denote the set of all the
multipermutations of $S$ that avoid $\rho$, and we let $\av(S;\rho)$ be the cardinality of
$\Av(S;\rho)$. Two multipermutations $\rho$ and $\sigma$ are \emph{$m$-equivalent}, denoted by $\rho
\msim \sigma$, if for every multiset $S$, $\av(S; \rho)$ equals $\av(S;\sigma)$.

The notion of $m$-equivalence can be straightforwardly extended to sets of patterns. For instance,
suppose that $P$ is a set of multipermutations. We let $\Av(S;P)$ be the set of multipermutations
of $S$ that avoid all the patterns contained in $P$, and we let $\av(S;P)$ be its cardinality. We
again call two sets $P$ and $Q$ of multipermutations $m$-equivalent, denoted $P\msim Q$, if
$\av(S;P)=\av(S;Q)$ for every multiset~$S$. To avoid clutter, we often omit nested braces and write,
e.g., $\Av(S;\pi,\rho)$ instead of $\Av(S;\{\pi,\rho\})$.

We may easily observe that each multipermutation $\rho$ is $m$-equivalent to its reversal $\rho^r$,
and that for every pair $\rho$ and $\sigma$ of $m$-equivalent patterns, we also have
$\rho^r\msim\sigma^r$ and $\rho^c\msim\sigma^c$. Moreover, these symmetry relations can be
generalized in an obvious manner to equivalences involving sets of patterns.

The notion of $m$-equivalence has been previously studied by Jel\'\i nek and Mansour~\cite{JM1},
who called it `strong equivalence'. They focused on the classification of this equivalence for
single patterns of fixed size, and they characterized the $m$-equivalence classes of patterns of
size at most six. From their results, we will use here the following
fact~\cite[Lemma~2.4]{JM1}.

\begin{fact}[Jel\'\i nek and Mansour~\cite{JM1}]\label{fac-1223}
For any $k$, all the patterns that consist of a single symbol `1', a single
symbol `3' and $k-2$ symbols `2' are $m$-equivalent.
\end{fact}

The main purpose of this paper is to classify the $m$-equivalence for sets of patterns containing a
pattern of size three and a pattern of size at most four. This extends earlier results concerning avoidance by multisets of a single permutation \cite{SavWilf} or word \cite{HMM} pattern of length three.  With the help of computer enumeration, we
identified the plausible $m$-equivalences and have managed to verify all of these $m$-equivalence.  This also yields all of the non-singleton (3,3) and (3,4) Wilf-equivalence classes for compositions, up to at most two sporadic cases. Many of our results are based on
arguments that generalize to larger patterns. However, to keep the presentation simple, we mostly
state our theorems and proofs for the special case of patterns of size up to four, which is our main
focus. We point out the possible generalizations separately as remarks.

\section{Avoidance results for multisets}

We may represent words of height $k$ and length $n$ as binary matrices with $k$ rows and $n$ columns
and exactly one $1$-cell in each column. We assume that the rows of a matrix are numbered
bottom-to-top, and the columns are numbered left-to-right. For a multipermutation $\sigma$
of height $k$, let $M(\sigma)$ be the $k\times n$ matrix with a 1-cell in row $i$ and column $j$ if
and only if the $j$-th letter of $\sigma$ is equal to $i$. For example,
$$M(31321)=\begin{bmatrix}
1& 0& 1& 0 & 0 \\
0& 0& 0& 1& 0 \\
0& 1& 0& 0 & 1 \\
\end{bmatrix}.$$
Conversely, if $M$ is a matrix with exactly one 1-cell in each column and at least one 1-cell in
each row, then there is a unique reduced multipermutation $\sigma$ such that $M=M(\sigma)$. If
there is no risk of confusion, we will identify a multipermutation $\sigma$ with its corresponding
matrix $M(\sigma)$, and we will say, for instance, that two matrices $M$ and $M'$ are
$m$-equivalent, if they represent two $m$-equivalent multipermutations.

The \emph{Ferrers diagram} (or \emph{Ferrers shape}) is an array of boxes (called cells) arranged
into down-justified columns, which have nonincreasing length. A \emph{filling} of a Ferrers diagram
is an assignment of zeros and ones into its cells. A filling is \emph{column-sparse} if every column
has at most one 1-cell. A filling is \emph{sparse} if every row and every column has at most one
1-cell. A \emph{transversal filling}, or a \emph{transversal}, is a filling in which every row and
every column has exactly one 1-cell. In this paper, we only deal with column-sparse fillings and
their various restrictions. We treat binary matrices, i.e. matrices containing only values 0 and 1,
as fillings of rectangular Ferrers diagrams.

The \emph{deletion of the $i$-th column} in a Ferrers diagram $F$ is the operation that removes from
$F$ all the boxes in the $i$-th column, and then shifts the boxes in columns $i+1, i+2,\dotsc$ one
unit to the left in order to fill the created gap. Note that the deletion transforms $F$ into a
smaller Ferrers diagram. Deletion of a row is defined analogously.

A filling $\phi$ of a Ferrers diagram $F$ \emph{contains} a binary matrix $M$ if $\phi$ can be
transformed into $M$ via a sequence of deletions of rows and columns, possibly followed by changing
some 1-cells of $\phi$ into 0-cells. If $\phi$ does not contain $M$, we say that $\phi$
\emph{avoids}~$M$. Notice that a
multipermutation $\sigma$ contains a multipermutation $\rho$ if and only if the matrix
$M(\sigma)$, understood as a Ferrers diagram, contains $M(\rho)$ in the sense defined above.

We say that two binary matrices $M_1$ and $M_2$ are \emph{strongly Ferrers-equivalent}, denoted
$M_1\sfsim M_2$, if for every Ferrers shape $F$, there is a bijection between $M_1$-avoiding and
$M_2$-avoiding column-sparse fillings of $F$ that preserves the number of 1-cells in each row and
column. We also say that $M_1$ and $M_2$ are \emph{Ferrers-equivalent for transversals}, or
\emph{FT-equivalent} for short, if for every Ferrers diagram $F$, the number of its $M_1$-avoiding
transversals is equal to the number of its $M_2$-avoiding transversals. This relation is denoted by
$M_1\ftsim M_2$.  For a pair of multipermutations $\sigma$ and $\rho$, we
will often write $\sigma\sfsim \rho$ or $\sigma\ftsim\rho$ for $M(\sigma)\sfsim M(\rho)$ and
$M(\sigma)\ftsim M(\rho)$, respectively. As with $m$-equivalence, we will also extend strong
Ferrers-equivalence and FT-equivalence from individual patterns to sets of patterns.

Clearly, if two patterns (or sets of patterns) $M_1$ and $M_2$ are strongly Ferrers-equivalent, then
they are also FT-equivalent. Moreover, as the next simple lemma shows, FT-equivalence can always be
extended from transversals to general sparse fillings, provided the two patterns have no zero rows or
columns.

\begin{lemma}\label{lem-sparse}
Suppose that $M_1$ and $M_2$ are FT-equivalent matrices, and that every row and column of $M_1$ and
of $M_2$ contains at least one 1-cell. Then for every Ferrers diagram $F$, there is a bijection
between $M_1$-avoiding and $M_2$-avoiding sparse fillings of $F$, which has the additional property
of preserving the number of 1-cells in each row and column of~$F$.
\end{lemma}
\begin{proof}
Let $F$ be a Ferrers shape, and let $\phi$ be an $M_1$-avoiding sparse filling of~$F$. We delete all
the rows and columns of $F$ that have no 1-cell in~$\phi$. This transforms the filling $\phi$ of the
diagram $F$ into an $M_1$-avoiding transversal $\phi'$ of a Ferrers diagram~$F'$. Since $M_1$ and
$M_2$ are FT-equivalent, there is a bijection $B$ that maps $M_1$-avoiding transversals of $F'$ into
$M_2$-avoiding transversals of~$F'$. We define $\psi':=B(\phi')$. We then reinsert the rows and
columns we deleted in the first step into $\psi'$, filling the newly inserted boxes by zeros. This
transforms $\psi'$ into a sparse filling $\psi$ of the original diagram~$F$. Since $M_2$ has a
1-cell in every row and column, the insertion of an all-zero row or column into $\psi'$ cannot
create an occurrence of~$M_2$. Thus, $\psi$ is an $M_2$-avoiding sparse filling of~$F$, and we may
easily observe that the transformation $\phi\mapsto\psi$ is the required bijection.
\end{proof}

Recall that a multipermutation $\sigma$ contains a multipermutation $\rho$ if and only if the matrix
$M(\sigma)$ contains $M(\rho)$. Thus, for a multiset
$S=1^{a_1}2^{a_2}\cdots k^{a_k}$ of size $n$ and a pattern $\rho$, there is a bijective
correspondence between the set $\Av(S;\rho)$ of $\rho$-avoiding multipermutations of $S$ and the set
of all the $M(\rho)$-avoiding matrices of shape $k\times n$ having exactly one 1-cell in each column
and exactly $a_i$ 1-cells in row $i$, for each $i\in[k]$. It follows that for two multipermutations
$\sigma$ and $\tau$, $\sigma\sfsim \tau$ implies $\sigma\msim\tau$.

If $\rho$ is a word and $k$ an integer, we denote by $\rho+k$ the word obtained by increasing
each letter of $\rho$ by~$k$. Recall that the height $\he(\rho)$ of a word $\rho$ is the maximum
value appearing in~$\rho$. For two words $\rho$ and $\tau$, we let their \emph{direct sum}
$\rho\oplus\tau$ be the concatenation of $\rho$ and $\tau+\he(\rho)$.
For instance, for $\rho=122$ and $\tau=312$, we have $\rho\oplus\tau=122534$. For a set of patterns
$P=\{\alpha^1, \alpha^2,\dotsc,\alpha^m\}$ and a pattern $\beta$, we write $P\oplus\beta$ as a
shorthand for the set $\{\alpha^1\oplus\beta, \alpha^2\oplus\beta,\dotsc,\alpha^m\oplus\beta\}$.

An important feature of the various flavors of Ferrers-equivalence is that they are closed with
respect to direct sums. This follows from a standard argument appearing, among others, in the works
of Backelin, West and Xin~\cite[Proposition 2.3]{bwx} or of Stankova and
West~\cite[Proposition~1]{SW} in the context of permutations, and later in the works of Jel{\'\i}nek
and Mansour~\cite[Lemma 2.1]{JM1} and~\cite[Lemma 14]{JM0} in the more general setting of words. We
omit repeating the argument here, and merely state the required result as a fact.

\begin{fact}[\cite{bwx,JM0,JM1,SW}]\label{fac-plus}
Let $P$ and $P'$ be two sets of multipermutations, and let $\rho$ be another multipermutation. If $P$
and $P'$ are strongly Ferrers-equivalent, then $P\oplus \rho$ and $P'\oplus \rho$ are also strongly
Ferrers-equivalent. Likewise, if $P$ and $P'$ are FT-equivalent, then $P\oplus \rho$ and $P'\oplus
\rho$ are also FT-equivalent.
\end{fact}

\subsection{Results based on Ferrers-equivalence arguments.}

We now state several known results on various forms of Ferrers-equivalence which will be useful for
our purposes. The first such result is the strong Ferrers-equivalence, for any $k$, of the
increasing pattern $12\dotsb k$ and the decreasing pattern $k(k-1)\dotsb 1$. This equivalence has
been established by Backelin et al.~\cite{bwx} for transversal fillings, and Krattenthaler~\cite{kra}
then obtained more general results which imply the strong Ferrers-equivalence of the two patterns.

\begin{fact}[Krattenthaler~\cite{kra}]\label{fac-diag} For any $k$, the pattern $12\dotsb k$ is
strongly Ferrers-equivalent to $k(k-1)\dotsb 1$.
\end{fact}

Another family of strongly Ferrers-equivalent patterns has been found by Jel\'\i nek and
Mansour~\cite[Lemma 39]{JM0}.

\begin{fact}[Jel\'\i nek and Mansour~\cite{JM0}]\label{fac-2p12q}
For any $i,j\ge 0$, the pattern $2^i12^j$ is strongly Ferrers-equivalent to $12^{i+j}$.
\end{fact}

The next result, due to Stankova and West~\cite{SW}, is specific to FT-equivalence.

\begin{fact}[Stankova and West~\cite{SW}]\label{fac-312}
The patterns $312$ and $231$ are FT-equivalent.
\end{fact}

In the statement of Fact~\ref{fac-312}, FT-equivalence cannot be directly replaced with strong
Ferrers-equivalence, as was pointed out by Guo et al.~\cite{GKZ}. However, Guo et
al.~\cite{GKZ} have found a different way of generalizing Fact~\ref{fac-312} to a strong
Ferrers-equivalence result, which we now state.

\begin{fact}[Guo et al.~\cite{GKZ}] We have the following strong Ferrers-equivalences for sets of
patterns:
\begin{itemize}
\item $\{231,221\}\sfsim\{312,212\}$ and
\item $\{231,121\}\sfsim\{312,211\}$.
\end{itemize}
\end{fact}

There is another, simpler way to translate an arbitrary FT-equivalence result into a strong
Ferrers-equivalence, which involves the pattern $11$. Clearly, a multipermutation $\sigma$ avoids
$11$ if and only if each of its elements has multiplicity 1, i.e., $\sigma$ is actually a
permutation. Similarly, a column-sparse filling of a Ferrers diagram avoids $M(11)$ if and only if
each row has at most one 1-cell, that is, the filling is sparse.

A pair of patterns $\sigma$, $\rho$ is said to be \emph{Wilf-equivalent}, denoted $\sigma\wsim\rho$,
if for every $n$, the number of permutations of $[n]$ that avoid $\sigma$ is the same as the number
of those that avoid~$\rho$. Intuitively speaking, strong Ferrers-equivalence refines $m$-equivalence
in the same way as FT-equivalence refines Wilf-equivalence. This intuition is made more rigorous by
the next easy observation, which can be easily deduced from the definitions and from
Lemma~\ref{lem-sparse}. We omit its proof.

\begin{observation}\label{obs-11}
For any two multipermutations $\sigma$ and $\rho$, if $\sigma\wsim\rho$ then
$\{\sigma,11\}\msim\{\rho,11\}$, and if $\sigma\ftsim\rho$ then $\{\sigma,11\}\sfsim\{\rho,11\}$.
\end{observation}

A similar observation states that $m$-equivalence, as well as strong Ferrers-equivalence, is
preserved when we add a pattern $1^r$ for any $r\ge 2$.

\begin{observation}\label{obs-1r} For any two patterns $\tau$ and $\tau'$ and for any $k\ge 2$,
$\tau\msim\tau'$
implies  $\{1^k,\tau\}\msim\{1^k,\tau'\}$, and $\tau\sfsim\tau'$
implies  $\{1^k,\tau\}\sfsim\{1^k,\tau'\}$.
\end{observation}

By combining the previous facts and observations, we obtain the following equivalences among pairs
involving a pattern of size 3 and a pattern of size~4.

\begin{proposition}\label{pro-ferrers} The following equivalences hold:
\begin{enumerate}[(a)]
\item $\{122,1111\}\msim\{212,1111\}$,
\item
$\{111,1223\}\msim\{111,1232\}\msim\{111,1322\}\msim\{111,2123\}\msim\{111,2132\}\msim\{111,2213\}$,
\item $\{111,1233\}\msim\{111,2133\}$,
\item
$\{111,1234\}\msim\{111,1243\}\msim\{111,1432\}\msim\{111,2134\}\msim\{111,2143\}\msim\{111,3214\}$,
\item
$\{123,1111\}\msim\{132,1111\}\msim\{213,1111\}$,
\item $\{123,1112\}\msim\{213,1112\}$,
\item $\{112,1234\}\msim\{112,2134\}\msim\{112,3214\}$,
\item $\{112,2314\}\msim\{112,3124\}$.
\end{enumerate}
\end{proposition}
\begin{proof}
Parts (a) to (e) all use Obs.~\ref{obs-1r} to conclude the $m$-equivalence in conjunction with Fact~\ref{fac-2p12q} for (a), Fact~\ref{fac-1223} for (b) and Facts~\ref{fac-plus} and \ref{fac-diag} for (c)-(e).  Part (f) and the equivalence $\{112,1234\}\msim\{112,3214\}$ make use of Fact~\ref{fac-diag} and Obs.~\ref{obs-1r} first and then Fact~\ref{fac-plus}.  The equivalence $\{112,1234\}\msim\{112,2134\}$ follows from the strong Ferrers-equivalence of $123$ and $213$, together with Obs.~\ref{obs-1r} and Fact~\ref{fac-plus}.  Finally, part (h) follows from combining Fact~\ref{fac-312}, Obs.~\ref{obs-11} and Fact~\ref{fac-plus} in that order.
\end{proof}

\subsection{Sorting minimal/maximal letter technique.}

In this subsection, we prove some equivalences by defining bijections which reorder the relevant pattern-avoiding multiset permutations, expressed as words.

\begin{theorem}\label{teo1}
The following pair of patterns are $m$-equivalent:
\begin{enumerate}
\item $\{111,1221\}\msim\{111,2112\}$,
\item $\{112,1211\}\msim\{121,1112\}$.
\end{enumerate}
\end{theorem}
\begin{proof}
(1) \, Fix a multiset $S=1^{b_1}\cdots k^{b_k}$, and let $\mathcal{A}$ and $\mathcal{B}$ denote the
sets $\Av(S;111,1221)$ and $\Av(S;111,2112)$, respectively. Given $\lambda \in
\mathcal{A}$, let $a_1>\cdots>a_r$ denote the set of letters within $\lambda$ which occur twice.  Let
$\rho \in \mathcal{B}$ be obtained from $\lambda$ by replacing each occurrence $a_i$ with
$a_{r+1-i}$ for $1 \leq i \leq r$, leaving all other letters unchanged in their positions.  Note that
these other letters must occur once and therefore cannot affect the avoidance of any of the
patterns we consider. It is then seen that the mapping $\lambda \mapsto \rho$ is a bijection between
$\mathcal{A}$ and $\mathcal{B}$, as desired.

(2) \, Equivalently, we show $\{122,2212\}\msim\{212,1222\}$. Let us write $P_A=\{122,2212\}$ and
$P_B=\{212,1222\}$. With $S$ as above, we will describe a bijection between $\Av(S;P_A)$
and $\Av(S;P_B)$. We proceed by induction on the number $k=\he(S)$. Clearly, if $k=1$,
the required bijection is the identity mapping, since there is only one multipermutation of $S$, and
it avoids all patterns with two or more symbols.

Define now the multiset $S'=1^{b_1}\cdots (k-1)^{b_{k-1}}$ obtained by removing all copies of $k$
from~$S$. By induction, there is a bijection $f_{k-1}$ between $\Av(S';P_A)$ and $\Av(S';P_B)$. Let
$\lambda'$ be a multipermutation from $\Av(S';P_A)$. We want to insert $b_k$ copies of $k$ into
$\lambda'$ to create a $P_A$-avoiding multipermutation $\lambda$ of~$S$. If $b_k=1$, then we may
insert the symbol $k$ in any position of $\lambda'$ while preserving $P_A$-avoidance. If $b_k=2$,
then we observe that one of the copies of $k$ must be the leftmost symbol of $\lambda$ in order
to preserve $P_A$-avoidance, while the other can be placed arbitrarily. If $b_k\ge 3$, then all the
symbols $k$ in $\lambda$ must appear consecutively at the leftmost $b_k$ positions.

Similarly, when extending a multipermutation $\rho'\in\Av(S';P_B)$ to a multipermutation
$\rho\in\Av(S;P_B)$, we proceed as follows: if $b_k=1$, the symbol $k$ can be placed arbitrarily, if
$b_k=2$, the only restriction is that the two symbols $k$ must appear consecutively, and if $b_k\ge
3$, then the symbols $k$ must form the leftmost $b_k$ symbols of~$\rho$.

The above description shows that in both the $P_A$-avoiding and the $P_B$-avoiding multipermutations
of $S$, the position of all the symbols $k$ is uniquely determined by the position of the rightmost
copy of $k$. This yields a straightforward bijection $f_k$ between $\Av(S;P_A)$ and $\Av(S;P_B)$,
defined as follows: fix a $\lambda\in\Av(S;P_A)$, remove from $\lambda$ all the occurrences of $k$ to
obtain a $\lambda'\in \Av(S';P_A)$, define $\rho'=f_{k-1}(\lambda')\in\Av(S';P_B)$, and finally, let
$\rho$ be the unique member of $\Av(S;P_B)$ in which the rightmost occurrence of $k$ appears at the
same position as the rightmost occurrence of $k$ in~$\lambda$. We easily see that this provides the
required bijection.
\end{proof}

\emph{Remark:} Extending the bijections described above shows more generally
 $$\{1^{i+1},\tau\}\csim\{1^{i+1},\tau^c\}, \qquad i \geq 2,$$
$$\{112,121^{i-1}\}\csim\{121,1^i2\}, \qquad i \geq 3,$$
where $\tau$ denotes any permutation of the multiset $1^i\cdots k^i$ and $\tau^c$ is the
complement of~$\tau$.

\subsection{Equivalences by analysis of active sites.}

In this subsection, we establish several equivalences by considering active sites within multiset
permutations and the associated generating trees for the patterns in question.  An active site of a
parent multipermutation is in general a position in which we may insert one or more copies of a
letter in producing its offspring without introducing a given set of patterns.  In some instances, it
will be convenient to modify this definition somewhat to accommodate the patterns in question.  Throughout, we consider permutations of multisets of $[k]$, though at times it will be more convenient notationally to insert either successively smaller or larger letters into a parent permutation in producing its offspring.  For examples of the generating tree method applied to the avoidance problem on ordinary
permutations, see, e.g., \cite{JW1,JW2}.

We first establish the equivalence of $\{112,2212\}$ and $\{121,2122\}$ via an active site analysis where we successively insert smaller and smaller letters into a parent permutation.

\begin{theorem}\label{112,2212}
The sets of patterns $\{112,2212\}$ and $\{121,2122\}$ are $m$-equivalent, that is
$$\{112,2212\}\msim\{121,2122\}.$$
\end{theorem}
\begin{proof}
Let $S=k^{a_1}(k-1)^{a_2}\cdots 1^{a_k}$, where $k \geq 1$ and $a_1,\ldots,a_k \geq 1$ are fixed.
We first enumerate members $\pi \in \Av(S;112,2212)$.  To do so, we consider the various partial
permutations $\pi_i \in \Av(S_i;112,2212)$, where $S_i=k^{a_1}\cdots (k-i+1)^{a_i}$ for $1 \leq
i \leq k$.  We form the permutations $\pi_{i+1}$ by inserting $a_{i+1}$ copies of $t=k-i$
appropriately into the $\pi_i$.  By an \emph{active site}, within a permutation $\pi_i$ of the stated
form where $1 \leq i \leq k$, we mean a position where one may insert a single copy of the letter $t$
without introducing an occurrence of 2212 (where one may assume $a_{k+1}=1$ in the case $i=k$).  It
is understood that if $a_{i+1}>1$, then all other letters $t$ are to be added at the very end of
$\pi_i$ in order to avoid 112.

Let $v_i=2$ if $a_i=1$ and $v_i=3$ if $a_i\geq 2$ for $1 \leq i \leq k$. We now show by induction on
$i$ that each $\pi_i \in \Av(S_i;112,2212)$ has exactly $s_i=1+\sum_{j=1}^i(v_j-1)$ active sites.
The $i=1$ case is apparent since there are $v_1$ (active) sites in the composition $\pi_1=k^{a_1}$
corresponding to the very first and very last positions of $\pi_1$ for all exponents $a_1$ and also
to the position directly after the first $k$ if $a_1\geq 2$.  Now assume that the hypothesis is true
in the $i$-case for some $1 \leq i<k$ and we show it holds in the $(i+1)$-case.  If $a_{i+1}=1$, then
a single $t$ can be inserted into any one of the sites of some $\pi_i$ without introducing either
pattern, and it is seen that regardless of where $t$ is inserted, the number of sites increases by
one (essentially,
one of the present sites is split into two).  Also, replacing $i$ with $i+1$ in $s_i$ raises its
value by one since $a_{i+1}=1$, which accounts for the additional site.  If $a_{i+1}>1$, then there
are two new sites introduced by the insertion of the letters $t$, i.e., one directly following the
leftmost added $t$ and another at the very end following the last $t$.  Since we have $s_{i+1}-s_i=2$
in this case, the induction is complete.

As all $\pi_i$ have the same number $s_i$ of active sites for each $i$ (with this number depending
only on $S$), we have that the number of possible $\pi_{i+1}$ is given by the product of $s_i$
with the number of $\pi_i$ for each $i$.  Thus, the number of possible $\pi=\pi_k \in
\Av(S;112,2212)$ of the stated form is given by $\prod_{i=1}^{k-1}s_i$.  A similar argument whose
main details we describe briefly shows that there are the same number of $\rho \in \Av(S;121,2122)$.
  Let $\rho_i\in \Av(S_i;121,2122)$ for $1\leq i \leq k$ and consider forming the $\rho_{i+1}$
from the various $\rho_i$ by inserting copies of $t$ appropriately.  Define active site analogously
except that now we insert all letters $u_{i+1}$ at the site as a single run (so as to avoid 121).
Reasoning by induction as before, one can show for each $i$ that there are $s_i$ sites in all
$\rho_i\in \Av(S_i;121,2122)$, which implies the same product formula as above for the number of
possible $\rho=\rho_k$.
\end{proof}

\begin{theorem}\label{112,2122}
 We have $\{112,2122\}\csim\{121,1222\}$.
\end{theorem}
\begin{proof}
We show that $|\Av(S;112,2122)|=|\Av(S;121,1222)|$, where $S$ is as in the preceding proof.  Let
$\pi_i$ denote an arbitrary member of $\Av(S_i;112,2122)$, where $S_i$ is as before.   By an
\emph{offspring} of $\pi_i$, we mean some $\pi_{i+1} \in \Av(S_{i+1};112,2122)$ that can be
obtained from $\pi_i$ by inserting $a_{i+1}$ copies of $t=k-i$ appropriately.  Define an \emph{active
site} of $\pi_i$ to be a position in which a (single) $t$ may be inserted without introducing 2122.
Note that an offspring of $\pi_i$ is produced when a single $t$ is added at an active site and all
other $t$ are added at the end.

Suppose $1 \leq i <k$ and that $\pi_i$ has exactly $\ell$ (active) sites.  We consider the nature of
the offspring of $\pi_i$ based on cases for the exponent $a_{i+1}$.  If $a_{i+1}=1$ or $2$, then one
may verify that each of the $\ell$ offspring of $\pi_i$ has $\ell+1$ or $\ell+2$ sites, respectively.
 If $a_{i+1}\geq 3$, first note in this case that every site of $\pi_i$ to the right of the leftmost
$t$ is lost, as all offspring in this case must end in at least two letters $t$.  Allowing the
leftmost $t$ to occur in each of the possible positions, it is seen that there is exactly one
offspring of $\pi_i$ that has $j$ sites for each $j \in [3,\ell+2]$.

Now consider forming $\rho_{i+1}\in \Av(S_{i+1};121,1222)$ for $1 \leq i <k$ from $\rho_{i}\in
\Av(S_{i};121,1222)$.  Define offspring and (active) site analogously as before.  Suppose that
$\rho_i$ has $\ell$ sites and we describe its offspring.  If $a_{i+1}=1$ or $2$, then it is seen
again that each offspring of $\rho_i$ has $\ell+1$ or $\ell+2$ sites, respectively.  If $a_{i+1}\geq
3$, then inserting the run $t^{a_{i+1}}$ into a site effectively nullifies all sites of $\rho_i$
occurring to the left of the run. Thus, the number of sites in the offspring of $\rho_i$ ranges from
$3$ to $\ell+2$ in this case.  Comparing the offspring of the various $\pi_i$ and $\rho_i$, one can
show by induction on $i$ (upon considering cases based on the exponent $a_{i+1}$) that the number of
members of $\Av(S_i;112,2122)$ and having exactly $r$ sites is the same as the corresponding number
of members of  $\Av(S_i;121,1222)$ for all $r \geq 2$.  Allowing $r$ to vary over all possible
values then implies the desired result.
\end{proof}

\begin{theorem}\label{112,2121}
We have $\{112,2121\}\csim\{121,1122\}$.
\end{theorem}
\begin{proof}
Let $\pi,\pi_i$, $\rho,\rho_i$, $S,S_i$ for $1 \leq i \leq k$ and offspring be defined as in
the proof of Theorem~\ref{112,2212}, but now in conjunction with the pattern sets $\{112,2121\}$ and
$\{121,1122\}$, respectively.  By an \emph{active site} in $\pi_i \in \Av(S_i;112,2121)$, we mean
a position in which one can insert a single letter $t$ such that no occurrence of $2121$ arises when
another copy of $t$ is appended to the end of the resulting multipermutation.  Let
$\text{act}(\lambda)$ denote the number of (active) sites of a multipermutation $\lambda$ and we will
make use of this notation in subsequent proofs.  Note that inserting a single $t$ (i.e., when
$a_{i+1}=1$) into any position of $\pi_i$ introduces neither 112 nor 2121 and changes the act
statistic value (always increasing it by one) if and only if the $t$ is inserted into a present site
of $\pi_i$.  On the other hand, if $a_{i+1}>1$, then inserting $t$ into a site $x$ of $\pi_i$ not the
last (and placing $a_{i+1}-1$ copies of $t$ at the end) nullifies all sites of $\pi_i$ to the right
of $x$, with $x$ effectively preserved; moreover, the site at the very end of $\pi_i$
is in essence replaced by a site at the very end of $\pi_{i+1}$.

Suppose now $\text{act}(\pi_i)=\ell$ where $1 \leq i<k$.  By the previous observations, if
$a_{i+1}=1$, then $\pi_i$ has $\ell$ offspring with $\ell+1$ sites and $m_i-\ell+1$ with $\ell$
sites, where $m_i=a_1+\cdots+a_i$.  If $a_{i+1}>1$, then it is seen that the set of act values in the
$\ell$ offspring of $\pi_i$ comprise the interval $[2,\ell+1]$.

Now define an active site in $\rho_i
\in \Av(S_i;121,1122)$ to be a position of $\rho_i$ in which one can insert a run of $t$ of length
two or more without introducing $1122$, with the corresponding statistic again denoted by act.  Upon
considering cases based on whether $a_{i+1}=1$ or $a_{i+1}>1$, one can show that the set of act
values of the offspring of $\rho_i$ where $\text{act}(\rho_i)=\ell$ is the same as those of the
offspring of $\pi_i$ above.  By induction on $i$ (the $i=1$ case trivial), the corresponding act
statistics on $\Av(S_i;112,2121)$ and $\Av(S_i;121,1122)$ are identically distributed for all
$1 \leq i \leq k$.  Taking $i=k$ in particular implies the desired equivalence of patterns.
\end{proof}

\begin{theorem}\label{112,2312}
We have $\{112,2312\}\csim\{121,1223\}\csim\{121,2213\}$.
\end{theorem}
\begin{proof}
We again make use of the same notation.  For the first pattern set, let us define an active site to be a
position of $\pi_i \in \Av(S_i;112,2312)$ in which a single $t$ may be inserted without
introducing $2312$, where any remaining $t$ must be added at the end of $\pi_i$.  Suppose
$\text{act}(\pi_i)=\ell$, where $1 \leq i<k$.  If $a_{i+1}=1$, then each offspring of $\pi_i$ is seen
to have $\ell+1$ sites.  On the other hand, if $a_{i+1}\geq 2$, then inserting the leftmost $t$ into
any site $y$ other than the last destroys all sites of $\pi_i$ to the right of $y$.  Note that $y$
itself is split into two sites, with a new site reemerging at the end of $\pi_{i+1}$ corresponding to
the final added $t$.  Thus as $y$ varies, one gets offspring whose act values comprise the interval
$[3,\ell+1]$.  If all of the letters $t$ are added at the very end of $\pi_i$, then every site of $\pi_i$
is preserved with each position directly following a $t$ active as well in this case, which implies
that the offspring will have $\ell+a_{i+1}$ sites altogether.

Now suppose $\rho_i\in \Av(S_i;121,1223)$ with $\text{act}(\rho_i)=\ell$ or $\gamma_i\in
\Av(S_i;121,2213)$ with $\text{act}(\gamma_i)=\ell$, where active sites are defined analogously.
One can show by comparable reasoning as before that if $a_{i+1}=1$, then the offspring of both
$\rho_i$ and $\gamma_i$ all have act values of $\ell+1$, whereas if $a_{i+1}\geq 2$, then the values
comprise the set $[3,\ell+1]\cup\{\ell+a_{i+1}\}$.  By induction on $i$ (the $i=1$ case trivial), it
is seen that the various act statistics defined on the sets consisting of the possible $\pi_i$,
$\rho_i$ or $\gamma_i$ are identically distributed for $1 \leq i \leq k$, which in particular implies
the desired equivalences.

It is also possible to establish the second equivalence via a bijection.  It is instructive to
describe such a bijection since it will be seen to preserve further statistics within the framework
of multiset equivalence.  It suffices to define a bijection $f$ between the set of permutations of
$1^{b_1}\cdots k^{b_k}$ that avoid $\{121,1223\}$ and those that avoid $\{121,2213\}$, where $k \geq
1$ and $b_1,\ldots,b_k\geq 1$ are fixed.  Suppose $\alpha=x_1\cdots x_m$, expressed as a word, is a
permutation belonging to the former set, where $m=b_1+\cdots+b_k$.  We first decompose $\alpha$ as
$\alpha=\alpha^{(1)}k\alpha^{(2)}$, where $\alpha^{(2)}$ contains no $k$.  Define
$\widetilde{\alpha}$ by $\widetilde{\alpha}=\text{rev}(\alpha^{(1)})k\alpha^{(2)}$, where
$\text{rev}(\alpha^{(1)})$ denotes the reversal of $\alpha^{(1)}$.  If $\alpha^{(2)}=\varnothing$,
then set $f(\alpha)=\widetilde{\alpha}$.  Otherwise, let $k_1$ denote the largest letter occurring in
$\alpha^{(2)}$ and suppose $\alpha^{(2)}=\alpha^{(3)}k_1\alpha^{(4)}$, where $\alpha^{(4)}$ contains
no $k_1$.

We now introduce the following definition.  Suppose $w=w_1w_2\cdots$ is a $k$-ary word and $i \in
[k]$.  Then we will refer to a (maximal) string of consecutive letters in $w$ all of which belong to
$[i,k]$ as an $i$-\emph{upper run} and a string all of whose letters belong to $[i-1]$ as an
$i$-\emph{lower run}.  We consider a left-to-right scan of the $k_1$-upper and $k_1$-lower runs of
$\lambda=\text{rev}(\alpha^{(1)})k\alpha^{(3)}k_1$.  Let $\delta_1,\ldots,\delta_t$ denote the
distinct $k_1$-lower runs in $\alpha^{(3)}$, where $t=0$ is possible. Then $\alpha^{(3)}$ can be decomposed as
$\alpha^{(3)}=\rho_0\delta_1\rho_1\cdots\delta_t\rho_t$ if $t>0$, with $\alpha^{(3)}=\rho_0$ if $t=0$, where $\rho_0,\ldots,\rho_t$ are $k_1$-upper runs with (only) $\rho_0$ and $\rho_t$ possibly empty.  Similarly, let $\text{rev}(\alpha^{(1)})=\tau_0\sigma_1\tau_1\cdots\sigma_s\tau_s$ if $s>0$, with $\text{rev}(\alpha^{(1)})=\tau_0$ if $s=0$, where $\sigma_1,\ldots,\sigma_s$ are $k_1$-lower runs, $\tau_0,\ldots,\tau_s$ are $k_1$-upper runs and $\tau_0,\tau_s$ are possibly empty.

We now define a multiset $\lambda^*$ derived from $\lambda$ as follows.  If $s>t>0$, then let $\lambda^*$ be defined as
$$\lambda^*=\tau_0\text{rev}(\delta_t)\tau_1\text{rev}(\delta_{t-1})\tau_2\cdots\text{rev}(\delta_1)\tau_t\sigma_1\tau_{t+1}\sigma_2\tau_{t+2}\cdots \sigma_{s-t}\tau_s k \rho_0\sigma_{s-t+1}\rho_1\cdots \sigma_s\rho_tk_1,$$
where it is seen that this definition may be extended to the case when $s=t>0$, with $\lambda^*=\lambda$ if $t=0$ for all $s$.  If $t>s>0$, then let $\lambda^*$ be given by
\begin{align*}
\lambda^*=&\tau_0\text{rev}(\delta_t)\tau_1\text{rev}(\delta_{t-1})\tau_2\cdots\text{rev}(\delta_{t-s+1})\tau_sk\rho_0\text{rev}(\delta_{t-s})\rho_1\text{rev}(\delta_{t-s-1})\rho_2\cdots
\text{rev}(\delta_1)\rho_{t-s}\\
&\sigma_1\rho_{t-s+1}\sigma_2\rho_{t-s+2}\cdots\sigma_s\rho_tk_1,
\end{align*}
which can be extended to the case when $t>s=0$.  That is, if $\lambda'$ is obtained from
$\lambda$ by replacing the runs $\delta_1,\ldots,\delta_t$ with
$\text{rev}(\delta_t),\ldots,\text{rev}(\delta_1)$ in that order, then $\lambda^*$ is obtained
from $\lambda'$ by repositioning the $k_1$-lower runs to the right of the last $k$ so that they now
occur prior to those to the left of this $k$, maintaining the order of the $k_1$-upper runs (as well
as the order of the letters within all runs).   Note that the $k_1$-upper runs of $\lambda^*$ are the same as those in $\lambda$, with the relative order of $k_1$-lower and $k_1$-upper runs in a left-to-right scan also seen to be the same.

If $\alpha^{(4)}$ is empty, then set $f(\alpha)=\lambda^*$.  Note that by reordering the $k_1$-lower
runs as described, we have eliminated any possible occurrences of $2213$ in which the $3$ can
correspond to the terminal $k_1$.  If $\alpha^{(4)}$ is non-empty, then let $k_2<k_1$ be the largest
letter of $\alpha^{(4)}$ and write $\alpha^{(4)}=\alpha^{(5)}k_2\alpha^{(6)}$, where $\alpha^{(6)}$
contains no $k_2$.  Then consider any $k_2$-lower runs within the $\alpha^{(5)}$ section of
$\tau=\lambda^*\alpha^{(5)}k_2$.  We arrange the $k_2$-lower runs of $\tau$ such that the reversals
of those in $\alpha^{(5)}$ occur (in reverse order) prior to the others, as we did with the
$\delta_i$'s above in $\lambda$.  If $\tau^*$ denotes the resulting word, then set $f(\alpha)=\tau^*$
if $\alpha^{(6)}=\varnothing$.

Otherwise, we continue in the manner described until $\alpha^{(2i)}=\varnothing$ for some $i \geq
4$, setting $f(\alpha)$ equal to the word that results after applying the procedure described above
for a final time.  By construction, it is seen that $f(\alpha)$ avoids 121, as does each word arising
from an intermediate step of the algorithm.  One may verify also that $f(\alpha)$ avoids 2213; note
that it suffices to check that $f(\alpha)$ contains no 2213 in which the 3 corresponds to a (strict)
right-left maximum.  To reverse $f$, consider successively the right-left maxima, starting with the
last letter and working back to the rightmost $k$, where we reverse each step of the algorithm
described above starting with the last.  Note that this may be done since the values of right-left
maxima are preserved by each step of the algorithm and hence by $f$.
\end{proof}

\emph{Remark:} From the preceding proof, we have in particular that the multiset equivalence of
$\{121,1223\}$ and $\{121,2213\}$ respects the last letter and right-left maxima statistics.
\medskip

\begin{theorem}\label{121,1322}
We have $\{121,1322\}\csim\{112,1232\}\csim\{112,2132\}$.
\end{theorem}
\begin{proof}
Proceeding as in the prior proof and using the same notation, consider permutations of
$k^{a_1}\cdots 1^{a_k}$ that avoid either $\{121,1322\}$, $\{112,1232\}$ or $\{112,2132\}$. If
$a_{i+1}=1$ and $\text{act}(\pi_i)=\ell$, then all offspring of $\pi_i$ have act value $\ell+1$ for
each pattern set.  If $a_{i+1}>1$, then the act values of the offspring of $\pi_i$ are seen to
comprise the interval $[a_{i+1}+1,a_{i+1}+\ell]$ in each case.  Note that when avoiding
$\{121,1322\}$, inserting $a_{i+1}>1$ copies of $t$ into a site $v$ of $\pi_i$ destroys all sites to
the left of $v$ while splitting $v$ into $a_{i+1}+1$ sites. If avoiding $\{112,1232\}$, inserting the
leftmost $t$ into a site of $\pi_i$, not the last, is seen to destroy all sites to its left when
$a_{i+1}>1$. Similar reasoning applies to $\{112,2132\}$ except that (non-terminal) sites to the right are
destroyed. Since the pattern sets obey the same rules with regard to the number of sites in
offspring, the result follows.
\end{proof}

\emph{Remark:} Extending the previous proof shows more generally
$$\{121,132^r\}\csim\{112,1232^{r-1}\}\csim\{112,2132^{r-1}\}, \qquad r \geq 1.$$

One can provide a proof of the following result analogous to the previous ones by modifying appropriately the definition of a site.  However, we find it more instructive to give a bijective argument which makes use of a certain encoding of the offspring and suggests how one might go about finding bijective proofs of other comparable results.

\begin{theorem}\label{teormulti1}
We have $\{112,1231\}\csim\{121,1132\}$.
\end{theorem}
\begin{proof}
Let  $\pi$ be a multipermutation of
the multiset $S=k^{a_1}\cdots 1^{a_k}$.  Given $1 \leq i \leq k$, let $S_i=k^{a_1}\cdots
(k-i+1)^{a_i}$ and $\pi_i \in \Av(S_i;112,1231)$ be obtained from $\pi$ by considering only the relative
positions of the parts $k,\ldots,k-i+1$.  We construct an encoding for $\pi$ as follows. Let
$\pi_i=\alpha_1\cdots\alpha_{m_i}$ as a word where $m_i=a_1+\cdots+a_i$ and $p_i$ denote the index
$r$ such that $\alpha_r<\alpha_{r+1}$ with $r$ maximal, if it exists (i.e., $p_i$ corresponds to the
rightmost ascent bottom of $\pi_i$), with $p_i=0$ otherwise (i.e., if $\pi_i$ is decreasing, perhaps
weakly). We seek to form $\pi_{i+1}$ from $\pi_{i}$ for $i\geq1$ by making an appropriate insertion of
the $a_{i+1}$ copies of $t=k-i$.

Note that in forming $\pi_{i+1}$ from $\pi_{i}$ when $p_{i}>0$ that there are exactly $p_{i}$
positions to the left of the rightmost ascent bottom of $\pi_{i}$ in which to insert $t$.
Observe further that if a $t$ is inserted anywhere to the left of $p_{i}$ in $\pi_{i}$, then
necessarily $a_{i+1}=1$, for otherwise $\pi_{i+1}$ would contain an occurrence of $1231$ as the remaining
copies of $t$ are forced to occur at the very end (in particular, to the right of the last ascent).

Let $\ell$ denote the smallest index $j>1$ such that $p_j>0$, assuming such $j$ exists.  Then
$p_1=\cdots=p_{\ell-1}=0$ implies $\pi_{\ell-1}$ is decreasing (i.e., $\pi_{\ell-1}=k^{a_1}\cdots
(u+1)^{a_{\ell-1}}$ where $u=k-\ell+1$).  Then $p_\ell>0$ and $\pi_\ell$ avoiding $112$ implies $\pi_\ell$ contains only one ascent, with $p_\ell$ determining the position of that ascent.
Note that $\pi_\ell$ is of the form $\pi_\ell=k^{a_1}\cdots x\cdots
(u+1)^{a_{\ell-1}}u^{a_\ell-1}$, where $x$ corresponds to the position of the leftmost
letter $u$ (i.e., $x$ corresponds to the ($p_\ell$)-th entry of $\pi_\ell$ where $1 \leq p_\ell
\leq m_{\ell-1}$).

Now assume $\ell \leq i \leq k-1$.  If it is the case that both $p_{i+1}=p_{i}+1$ and $a_{i+1}=1$, then we  must also specify, in addition to the value of $p_{i+1}$, some element of the set $[p_{i}+1]$, which gives the entry number of the only $t$ in $\pi_{i+1}$.  Note that when $a_{i+1}=1$, we have $p_{i+1}=p_i+1$ if and only if the $t$ is inserted to the left of or within the rightmost ascent of $\pi_i$, as the position of the rightmost ascent bottom is shifted to the right by one place in this case.  Thus, the position number of $t$ must belong to $[p_i+1]$. Let $\rho$ denote the vector $(p_1,\ldots,p_k)$ consisting of the various $p_i$ values, where in addition an element of $[p_{i}+1]$ is specified parenthetically in the $i$-th component for each $i \in [\ell,k-1]$ such that $p_{i+1}=p_{i}+1$ and $a_{i+1}=1$.  One can verify that the sequence $p_{i}$ satisfies the following succession rules for $\ell  \leq i \leq k-1$:
$$p_{i+1}=\begin{cases}
{\displaystyle p_{i}+1}, &\emph{if~the~(single)}~t~\emph{is~inserted anywhere to the left of position}~p_{i}~\emph{in}~\pi_{i};\\
{\displaystyle p_{i}+s}, &\emph{if~the~leftmost}~t~\emph{is~inserted~in}~s\emph{-th}~\emph{position~to~the~right~of~position}~p_{i}~\emph{in}~\pi_{i},\\
{\displaystyle ~} &\emph{but~not~at the~very~end~of}~\pi_{i};\\
{\displaystyle p_{i}},  &\emph{if~the~run~of}~t~\emph{is~inserted~at~the~very~end~of}~\pi_{i}.
\end{cases}$$
From this, it is seen
that $\pi$ can be reconstructed from $\rho$, starting with a run of  $k$ and successively
inserting $k-1,\ldots,1$.  Note that in cases when there is no additional specified member of
$[p_{i}+1]$, the position of the added $t$ is determined completely by $p_{i+1}$ alone (i.e., when
$p_{i+1}\neq p_{i}+1$ or when $p_{i+1}=p_{i}+1$ and $a_{i+1}>1$).

Now let $\lambda$ be a multipermutation of $S$ that avoids $\{121,1132\}$.
Let $\lambda_1=k^{a_1}$ and $\lambda_{i+1}\in \Av(S_{i+1};121,1132)$ for $i\geq1$ be formed from
$\lambda_{i}\in \Av(S_{i};121,1132)$ by inserting $a_{i+1}$ copies of $t$ as before. If $1 \leq i
\leq k$, then let $\lambda_i=\beta_1\cdots \beta_{m_i}$ as a word and let $p_i'$ denote the index $r$
such that $\beta_r>\beta_{r+1}$ with $r$ maximal, if it exists (i.e., $p_i'$ corresponds to the
rightmost descent top of $\lambda_i$), with $p_i'=0$ otherwise (i.e., if $\lambda_i$ is increasing,
perhaps weakly).  Suppose $\ell>1$ is determined by $p_1'=\cdots=p_{\ell-1}'=0$, with $p_\ell'>0$
(assuming such $\ell$ exists).  Then the position of the (only) run of $u$ in $\lambda_\ell$ is
such that the first $u$ corresponds to the ($p_{\ell}'+1$)-st entry, where $1 \leq p_\ell' \leq
m_{\ell-1}$.

Suppose $\ell \leq i \leq k-1$.  If $p_{i+1}'=p_{i}'$, then the run of $t$ is inserted so that it
separates the letters of the rightmost descent of $\lambda_{i}$.  If $a_{i+1}=1$ and
$p_{i+1}'=p_{i}'+1$, then additionally we must specify some number $x$ in $[p_{i}'+1]$, which
indicates the position of the lone inserted $t$ letter if $x<p_{i}'+1$, with $t$ inserted directly
following the rightmost descent bottom of $\lambda_{i}$ if $x=p_{i}'+1$.  Let
$\rho'=(p_1',\ldots,p_k')$, where in addition an element of $[p_{i}'+1]$ is specified
parenthetically in the $i$-th component for each $i \in [\ell,k-1]$ satisfying the stated
requirements.  Note that $p_{i+1}'$ for $\ell\leq i \leq k-1$ satisfies
$$p_{i+1}'=\begin{cases}
{\displaystyle p_{i}'+1}, &\emph{if~the~(single)}~t~\emph{is~inserted to the left of position}~p_{i}'~\emph{in}~\lambda_{i};\\
{\displaystyle p_{i}'+s}, &\emph{if~the}~t~\emph{run~is~inserted~in}~s\emph{-th}~\emph{position~to~the~right~of~the}~(p_{i}'+1)\emph{-st~entry}\\
{\displaystyle ~} &\emph{of}~\lambda_{i};\\
{\displaystyle p_{i}'},  &\emph{if~the}~t~\emph{run~is~inserted~between~the}~(p_{i}')\emph{-th~and}~ (p_{i}'+1)\emph{-st~entries~of}~\lambda_{i}.
\end{cases}$$
From this, we see that $\lambda$ can be recovered from $\rho'$.  Note that the $p_i$ and $p_i'$ for
$i>\ell$ satisfy equivalent recurrences and hence yield the same set of possible vectors $\rho$ and
$\rho'$ (along with any parenthetical elements).

Thus, one may define a bijection between multipermutations avoiding $\{112,1231\}$ and those
avoiding $\{121,1132\}$ by starting with $\pi$ from the former set and producing its vector $\rho$.
Then read the entries from $\rho$ as a vector $\rho'$ for constructing members of the latter set
using the multiset of parts from $\pi$.  This yields $\pi'$ belonging to the latter set and the
mapping $\pi \mapsto \pi'$ is seen to be a bijection.  Note that this mapping preserves all part size
multiplicities, and in particular, the number of parts.
\end{proof}

\begin{theorem}\label{123,1121}
 We have $\{123,1121\}\csim\{132,1211\}\csim\{213,1121\}$.\label{1231121th}
\end{theorem}
\begin{proof}
Let $S=k^{a_k}\cdots1^{a_1}$ and $S_i=k^{a_k}\cdots i^{a_i}$ for $1 \leq i \leq k$.  We show
that permutations of $S$ avoiding any of the three sets of patterns are equinumerous.  We first
consider $\{123,1121\}$.  Consider forming members $\pi_i \in \Av(S_i;123,1121)$ from $\pi_{i+1}
\in \Av(S_{i+1};123,1121)$ for $1 \leq i<k$ by inserting $i$'s appropriately into $\pi_{i+1}$.
Note that all of the $i$'s must be inserted to the right of the rightmost ascent bottom of
$\pi_{i+1}$, lest a $123$ would arise.  Furthermore, if $a_i>1$, the $i$'s must be inserted either as
a single run or as two runs in which the first is a single $i$.  Suppose
$\text{act}(\pi_{i+1})=\ell$, meaning that the length of its final decreasing sequence is $\ell-1$.
If $a_{i}=1$, then $\pi_{i+1}$ has offspring with $2,3,\ldots,\ell+1$ sites.  If $a_i>1$, then for
each $j \in [2,\ell]$, we have that $\pi_{i+1}$ has $\ell+1-j$ offspring with $j$ sites and not
ending in $i$.  These offspring are obtained by inserting $a_i-1$ letters $i$ in the $j$-th rightmost
site and a single $i$ in any of the $\ell+1-j$ sites to the left of and including the first insertion
point.  On the other hand, there are $\ell$ offspring of $\pi_{i+1}$ ending in $i$, with their act
values comprising the interval $[a_i+1,a_i+\ell]$.  To see this, note that the final run of $i$
contributes $a_i-1$ sites in addition to those occurring at the point just when the leftmost $i$ has
been inserted into $\pi_{i+1}$.

In the case of avoiding $\{132,1211\}$, a site corresponds to a position to the right of the
rightmost descent top.  Thus $\text{act}(\pi)=\ell$ in this case means that the final increasing run
is of length $\ell-1$.  A similar analysis as before which we leave to the reader reveals that the
same rules are followed concerning the number of sites in offspring and yields the first equivalence.

We use similar notation as before when discussing members of $\Av(S;213,1121)$.  Note that a site of
$\pi_i \in \Av(S_i;213,1121)$ where $i>1$ corresponds either to the first or last position of
$\pi_i$ or to a position directly following any letter $x$ such that $\pi_i=\alpha x \beta$, where
$\beta$ is nonempty and $\min(\alpha\cup\{x\}) \geq \max(\beta)$.  Then inserting an $i$ into a site
of some $\pi_{i+1}$, not the last, destroys all sites to its right except for the last, whereas
inserting an $i$ at the end clearly adds a site.  Furthermore, if $a_i>1$ and
$\text{act}(\pi_{i+1})=\ell$, then there are $\ell-j$ offspring of $\pi_{i+1}$ having $j$ sites and
two runs of $i$ for each $j \in [2,\ell]$, where the second run is not at the end.  The offspring for
which there are two runs of $i$ where the second run occurs at the end have act values comprising
$[a_i+1,a_i+\ell-1]$, whereas those containing a single run of $i$ have values in
$[2,\ell]\cup\{a_i+\ell\}$. Combining these observations implies the second equivalence and completes
the proof.
\end{proof}

\begin{theorem}\label{123,1112}
We have $\{123,1112\}\csim\{132,1121\}$ and $\{123,1211\}\csim\{213,1211\}$.
\end{theorem}
\begin{proof}
We use the same notation as in the proof of Theorem \ref{1231121th} above and form permutations
$\pi_i \in \Av(S_i;123,1112)$ from $\pi_{i+1}\in \Av(S_{i+1};123,1112)$ by inserting $a_i$
letters $i$ and similarly for $\{132,1121\}$.  Note that sites of $\pi_{i+1}$ in the case of avoiding
$\{123,1112\}$ or $\{132,1121\}$ correspond to possible points of insertion to the right of the
rightmost ascent bottom or descent top, respectively.  Suppose $\text{act}(\pi_{i+1})=\ell$ in both
and we count sites in the offspring.  If $a_i=1$, then in either case the offspring have
$2,3,\ldots,\ell+1$ sites. If $a_i>1$, then for $\{123,1112\}$, we consider cases based on whether or
not there are two letters $i$ occurring in positions to the left of the rightmost letter in
$[i+1,k]$.  If so, then for each $j \in \{0,1,\ldots,\ell-2\}$, there are $\ell-1-j$ offspring that
have $a_i+j$ sites, upon inserting an $i$ in the $(\ell-1-j)$-th site from the left, inserting a
second $i$ anywhere to its left in any site and then adding the remaining $a_i-2$ letters $i$ at the
end of $\pi_{i+1}$.  If not, then at most a single $i$ occurs in a site other than the very last,
which results in offspring with act values in $[a_i+1,a_i+\ell]$.  Combining the previous cases, we
have for $a_i>1$,
$$\ell\rightsquigarrow a_i~(\ell-1),~a_i+j~(\ell-j) \text{ for } j \in [\ell-1],~a_i+\ell~(1),
\qquad \ell \geq 2,     \qquad(*)$$
where the multiplicities of the offspring (having the specified number of sites) are denoted in
parentheses.

On the other hand, if $\text{act}(\pi_{i+1})=\ell$ for $\pi_{i+1}\in \Av(S_{i+1};132,1121)$ and
$a_i>1$, we consider cases on the offspring $\pi_i$ based on whether or not two or more $i$'s occur
in positions prior to the rightmost letter in $[i+1,k]$.  If so, then $\pi_i$ cannot end in $i$ and
consider further cases on the position of the second leftmost inserted $i$.  If it occurs in the
$j$-th leftmost site of $\pi_{i+1}$ for $1 \leq j \leq \ell-1$, then there are $j-1$ offspring $\rho$
for which $\text{act}(\rho)=a_i+\ell-j$ and a single $\rho$ with $\text{act}(\rho)=a_i+\ell-j+1$.  If
there is at most one $i$ occurring prior to the rightmost letter in $[i+1,k]$, then it is seen that
there are $\ell-1$ possible offspring $\rho$ with $\text{act}(\rho)=a_i$ and a single $\rho$ with
$\text{act}(\rho)=a_i+1$.  Combining these cases demonstrates that the offspring of $\pi_{i+1}$ have
the same distribution of act values as those in (*) above when $a_i>1$, which implies the first
equivalence.

Similar reasoning applies to the second equivalence.  Note that sites in the case of avoiding
$\{213,1211\}$ correspond to positions  in the parent permutation $\pi$ wherein one can insert a
single $i$ without introducing $213$ and therefore correspond to the very beginning or end of $\pi$
or directly after $x$, where $x$ is such $\pi=AxB$ and $\min(A\cup\{x\}) \geq \max(B)$.  Further, the
$i$'s must occur as one or two runs for both pattern sets, where a second run is a single letter, for
otherwise $1211$ would arise.  For both sets of patterns, one may verify the following succession
rules:  $\ell \rightsquigarrow 2~(\ell-1),~a_i+\ell~(1),~j~(\ell-j+2) \text{ for } j \in [3,\ell+1]$
if $a_i>1$, with $\ell\rightsquigarrow 2,3,\ldots, \ell+1$ if $a_i=1$, which yields the second
equivalence.
\end{proof}

Combining the first equivalence of Theorem~\ref{123,1112} with a symmetric version of the
equivalence $\{123,1222\}\msim\{132,1222\}$ from Proposition~\ref{pro-ferrers} yields
$\{123,1112\}\msim\{132,1121\}\msim\{213,1112\}$.

\subsection{Equivalences involving doubly labeled offspring.}

To establish the equivalences in this subsection, it will be convenient to label the offspring by a
vector $(a,b)$ tracking two kinds of active sites.

\begin{theorem}\label{121,1243}
The following equivalences hold:  $\{121,1243\}\csim\{121,2143\}$,
$\{112,1243\}\csim\{112,2143\}$ and $\{121,1234\}\csim\{121,2134\}$.
\end{theorem}
\begin{proof}
Let $S_i=k^{a_k}\cdots i^{a_i}$ for $1 \leq i \leq k$. For the first equivalence, we form $\pi_i
\in \Av(S_i;\alpha,\beta)$ from $\pi_{i+1}\in \Av(S_{i+1};\alpha,\beta)$ by inserting $a_i$
letters $i$ for $1 \leq i <k$, where $\{\alpha,\beta\}=\{121,1243\}$ or $\{121,2143\}$.  We first
consider $\{121,1243\}$.  Suppose $\rho \in \Av(S_i;121,1243)$ can be written as $\rho=\gamma xy
\gamma'$, where $x>y$ and $y\gamma'$ is (weakly) increasing, i.e., $x$ is the rightmost descent top
of $\rho$.  Label $\rho$ by $(a,b)$, where $a$ denotes the number of active sites to the left of $x$
and $b$ is the number to the right of $x$.  We will refer to the sites accounted for by $a$ and $b$
as \emph{primary} and \emph{secondary}, respectively.  Furthermore, we take $a=0$ if $\rho$ contains
no descents, i.e., if $\rho$ is increasing. Given $\pi_{i+1}\in \Av(S_{i+1};121,1243)$ with label
$(a,b)$, we determine the labels of its $a+b$ offspring $\rho$.  If $a>0$, then
$$(a,b)\rightsquigarrow(1,b),(2,b),\ldots,(a,b),(a,a_i+b),(a+1,a_i+b-1),\ldots,(a+b-1,a_i+1).$$

To establish this rule, first note that all $i$'s must be inserted as a single run in $\pi_{i+1}$.
Furthermore, if $xy$ represents the rightmost descent of $\pi_{i+1}$, then inserting one or more
$i$'s in the $j$-th site to the left of $x$ destroys all sites to the left.  The $j$-th primary site
itself is in essence preserved since the position directly following the final added $i$ is still
active.  Letting $j$ range over $[a]$ accounts for $(1,b),\ldots,(a,b)$.  Otherwise, the $i$'s are
added to the right of $x$, i.e., within the final increasing sequence of $\pi_{i+1}$.  In this case,
if the $i$'s are inserted within the $j$-th position to the right of $x$, where $1 \leq j \leq b$,
then the rightmost descent shifts to the right by $j-1$ positions and thus the first $j-1$ secondary
sites become primary.  Further, it is seen that there are now $a_i+b-j+1$ secondary sites.  Letting
$j$ vary over $[b]$ then accounts for $(a,a_i+b),(a+1,a_i+b-1),\ldots,(a+b-1,a_i+1)$.

On the other hand, if $a=0$, then $b=a_{i+1}+\cdots+a_k+1$, with $\pi_{i+1}$ increasing.  In this case,
if the run of $i$ is added directly following $z$ in $\pi_{i+1}$, then all positions to the left
(right) of $z$ are primary (secondary).  This implies the succession rule
$$(0,b)\rightsquigarrow(0,a_i+b),(1,a_i+b-1),\ldots,(b-1,a_i+1).$$
One may verify that the same succession rules for $(a,b)$ when $a>0$ and $a=0$ are followed when
avoiding $\{121,2143\}$.  Note that insertion of $i$ into the $j$-th site to the left of $x$ now
destroys all primary sites to its right (instead of to its left).

For the second equivalence, we make use of the same labels $(a,b)$ as before, but now in conjunction
with the relevant patterns.  If $a>0$, then we have
$$(a,b)\rightsquigarrow(b,a_i),(b+1,a_i),\ldots,(a+b-1,a_i),(a+b,a_i)~(b-1~\text{times}),(a+b-1,
a_i+1), \qquad a_i>1,$$
with
$$(a,b)\rightsquigarrow(1,b),(2,b),\ldots,(a,b),(a,b+1),(a+1,b),\ldots,(a+b-1,2), \qquad a_i=1.$$
If $a=0$, then $b=a_{i+1}+\cdots+a_k+1=k-i+1$ since $a_{i+1}=\cdots=a_k=1$ in order for $\pi_{i+1}$
to be increasing, for otherwise a $112$ would arise.  Thus, we have the rules
$$(0,b)=(0,k-i+1)\rightsquigarrow (b-1,a_i+1)~\text{and}~(b,a_i)~(b-1~\text{times}), \qquad a_i>1,$$
and
$$(0,b)=(0,k-i+1)\rightsquigarrow(0,b+1),(1,b),\ldots,(b-1,2), \qquad a_i=1.$$
One may verify that both pattern sets $\{112,1243\}$ and $\{112,2143\}$ obey the preceding
succession rules.

For the final equivalence, we show alternatively $\{212,1234\}\csim\{212,1243\}$.  We form $\pi_i
\in \Av(S_i;\alpha,\beta)$ from $\pi_{i-1} \in \Av(S_{i-1};\alpha,\beta)$ for $i>1$, where
$S_i=1^{b_1}\cdots i^{b_i}$ and $\{\alpha,\beta\}=\{212,1234\}$ or $\{212,1243\}$.  We use the
labels $(a,b)$, where $a$ and $b$ now denote the number of sites to the right of or to the left of
the leftmost ascent top $x$, respectively.  Sites of the former or of the latter kind will be
described as primary or secondary, respectively, and we take $a=0$ if a permutation is (weakly)
decreasing.  Note that, in the case of avoiding $\{212,1234\}$, a position to the right of $x$ is a
(primary) site if it occurs to the left of all letters $z$  playing the role of a $3$ in some
occurrence of $123$ and that all positions to the left of $x$ are (secondary) sites.  We have the
following succession rule if
$a>0$:$$(a,b)\rightsquigarrow(1,b),(2,b),\ldots,(a,b),(a+b+b_i-2,2),(a+b+b_i-3,3),\ldots,(a+b_i,b),(a
,b+b_i).$$
To see this, note that the first $a$ offspring of $\pi_{i-1}$ listed account for the case when the
$i$'s are placed in a primary site (which destroys all sites to the right).  If the letters $i$ are
placed at the beginning of $\pi_{i-1}$, then one gets $(a,b+b_i)$, whereas if the $i$'s are placed in
the $j$-th site from the left where $2 \leq j \leq b$, then one gets $(a+b+b_i-j,j)$ since in this
case the final $b-j$ secondary sites become primary.  This accounts for the remaining offspring.  If
$a=0$, then $b=b_1+\cdots+b_{i-1}+1$ and $\pi_{i-1}$ is decreasing, which implies
$$(0,b)\rightsquigarrow (0,b+b_i),(2,b+b_i-2),(3,b+b_i-3),\ldots,(b,b_i).$$
Note that $(0,b+b_i)$ accounts for the case when all $i$'s are added at the beginning, whereas
inserting the $i$'s directly after the $j$-th letter of $\pi_{i-1}$ from the left gives
$(j+1,b+b_i-j-1)$ for $1 \leq j \leq b-1$.  By similar arguments, one can show that $\{212,1243\}$
follows the same succession rules, which implies the third equivalence.
 \end{proof}

\subsection{Equivalence of \texorpdfstring{$\{132,2213\}$}{\{132,2213\}} and
\texorpdfstring{$\{213,1322\}$}{\{213,1322\}}.}

We apply an active site analysis in proving the equivalence of $\{132,2213\}$ and $\{213,1322\}$ for
multisets which requires a further enumeration of several classes of offspring arising from a parent
permutation having a fixed number $\ell$ of sites.

\begin{theorem}\label{132,2213}
 We have $\{132,2213\}\csim\{213,1322\}$.
\end{theorem}
\begin{proof}
We show $|\Av(S;132,2213)|=|\Av(S;213,1322)|$, where $S=1^{b_1}\cdots k^{b_k}$.  Let $\pi_i$
denote a permutation of $S_i=1^{b_1}\cdots i^{b_i}$ for $1 \leq i \leq k$ avoiding either pattern
set.  We form all possible $\pi_i$ from $\pi_{i-1}$ for $i>1$ by inserting $b_i$ copies of $i$.
Suppose $\text{act}(\pi_{i-1})=\ell$ and we consider the act values of its offspring.  It is seen
that both pattern sets obey the succession rule $\ell \rightsquigarrow 2,\ldots,\ell+1$ if $b_i=1$;
henceforth, assume $b_i>1$.

We now consider the offspring $\pi_i\in \Av(S_i;132,2213)$ of $\pi_{i-1}\in
\Av(S_{i-1};132,2213)$.  First note a letter $i$ may be inserted into $\pi_{i-1}$ at the very
beginning or directly following $x$ such that $\pi_{i-1}=AxB$, where $\min(A\cup\{x\})\geq \max(B)$,
$A\cup\{x\}$ avoids $221$ and $B$ is possibly empty. Suppose that the sites of $\pi_{i-1}$ are
labeled $1$ to $\ell$, starting with the leftmost.  Let $j$ and $r$ denote respectively the numbers
of the sites in which the leftmost and second leftmost letter $i$ are inserted, where $1 \leq j \leq
r \leq \ell$.  To enumerate the offspring $\pi_i$, we consider several cases on $j$ and $r$.  First
suppose $j=r=1$, and let $s$ denote the number of additional letters $i$  inserted into this site,
where $0 \leq s \leq b_i-2$.  Note that only the beginning or positions coming directly after the
added $i$ are active in the offspring in this case, due to avoidance of $2213$.  Thus, there are
$\binom{b_i-s+\ell-4}{\ell-2}$ possible offspring $\rho$ with $\text{act}(\rho)=s+3$ for each $s$ if
$\ell \geq 2$, with a single offspring corresponding to $s=b_i-2$ if $\ell=1$.

If $j=1$ and $r>1$, then all possible offspring $\rho$ containing at least three runs of $i$ have
act value $2$.  To see this, note that any position of $\rho$ beyond the first letter in $[i-1]$ and
to the left of the rightmost added $i$ cannot be a site (due to $132$), while any position to the
right of, but not directly following, the last $i$ in the second run of $i$ isn't a site either (due
to $2213$).  Thus, the only possible sites are directly before or after the initial $i$.  If $\rho$
has only two runs, then the position directly following the last $i$ is also a site so that $\rho$
has three sites in this case.  Thus, there are $\binom{b_i-2+\ell-r}{\ell-r}-1$ possible $\rho$ with
$\text{act}(\rho)=2$, and one with $\text{act}(\rho)=3$.  Now suppose $j=r>1$.  By similar reasoning
as before, we have $\text{act}(\rho)=1$ for which there are at least two runs, as only the initial
position can then be a site.  If there is a single run of $i$ in $\rho$, then $\text{act}(\rho)=2$ in
this case since there is also a site directly following the last $i$.  Thus, there are
$\binom{b_i-2+\ell-r}{\ell-r}-1$ possible $\rho$ with $\text{act}(\rho)=1$, and one with
$\text{act}(\rho)=2$.  A similar analysis reveals that this is also the case for all $j$ and $r$ such
that $1 < j < r \leq \ell$.

We now determine the number of offspring corresponding to each act value.  Combining the prior
cases, and considering all possible $j$ for each $r$ in the last case, we have that the total number
of $\rho$ for which $\text{act}(\rho)=1$ is given by
$$\sum_{r=2}^\ell
(r-1)\left(\binom{b_i-2+\ell-r}{\ell-r}-1\right)=\sum_{r=0}^{\ell-2}(\ell-r-1)\binom{b_i-2+r}{r}
-\binom{\ell}{2}.$$
Simplifying further, and making use of the upper summation formula $\sum_{k=m}^n\binom{k}{m}=\binom{n+1}{m+1}$ (see,
e.g., \cite[p.~174]{GKP}), gives
\begin{align*}
&(\ell-1)\binom{b_i+\ell-3}{b_i-1}-\binom{\ell}{2}-\sum_{r=0}^{\ell-2}r\binom{b_i-2+r}{b_i-2}\\
&=(\ell-1)\binom{b_i+\ell-3}{b_i-1}-\binom{\ell}{2}-(b_i-1)\sum_{r=1}^{\ell-2}\binom{b_i-2+r}{b_i-1}\\
&=(\ell-1)\binom{b_i+\ell-3}{b_i-1}-(b_i-1)\binom{b_i+\ell-3}{b_i}-\binom{\ell}{2}.
\end{align*}
Also, combining the various cases above implies that the number of $\rho$ for which
$\text{act}(\rho)=2$ is given by
$$\ell-1+\binom{\ell-1}{2}+\sum_{r=2}^{\ell}\left(\binom{b_i-2+\ell-r}{\ell-r}-1\right)=\binom{
b_i+\ell-3}{b_i-1}+\binom{\ell-1}{2},$$
and the number for which $\text{act}(\rho)=3$ by $\binom{b_i+\ell-4}{\ell-2}+\ell-1$, assuming $\ell
\geq 2$.  Further, if $b_i \geq 3$, then the first case above yields $\binom{b_i-j+\ell-1}{\ell-2}$
possible $\rho$ such that $\text{act}(\rho)=j$ for $4 \leq j \leq b_i+1$.  Therefore, if $\ell, b_i
\geq 2$, we have the following succession rules:
$$\ell \rightsquigarrow\begin{cases}
{\displaystyle 1~~\left((\ell-1)\binom{b_i+\ell-3}{b_i-1}-(b_i-1)\binom{b_i+\ell-3}{b_i}-\binom{\ell}{2}\right)},\\
{\displaystyle 2~~\left(\binom{b_i+\ell-3}{b_i-1}+\binom{\ell-1}{2}\right)},\\
{\displaystyle 3~~\left(\binom{b_i+\ell-4}{b_i-2}+\ell-1\right)},\\
{\displaystyle j\in[4,b_i+1]~~\left(\binom{b_i+\ell-j-1}{b_i-j+1}\right)},
\end{cases}$$
where the number of offspring corresponding to the given act value is given parenthetically.  Note
that the last case does not apply if $b_i=2$.  Also, from the first case above, we have $1
\rightsquigarrow b_i+1$ for all $b_i \geq 2$.

We now consider $\{213,1322\}$ and form $\pi_i \in \Av(S_i;213,1322)$ from $\pi_{i-1}\in
\Av(S_{i-1};213,1322)$.  First note that a site of $\pi_{i-1}$ is a position, including at
the very beginning, that lies within the initial increasing run with the restriction that inserting
$i$ does not introduce $1322$.  To elucidate, suppose $\pi_{i-1}=a_1\cdots a_m \pi'$, where
$a_1\leq\cdots \leq  a_{m-1}$ and $a_{m-1}>a_m$.  Note that any letter for which there are at least
two occurrences to the right of $a_{m-1}$ cannot belong to $[a_1+1,a_{m-1}-1]$, for otherwise
$\pi_{i-1}$ would contain 1322 of the form $a_1a_mxx$.  If some letter in $[a_{m-1}+1,i-1]$ is
repeated, then $\text{act}(\pi_{i-1})=1$, so assume that this is not the case.  On the other hand,
suppose that a letter $a_i=s>a_1$ where $i<m$ occurs at least twice. For each such $s$, define $p=p_s
\geq 2$ such that $s=a_p$, with $p$ the second-to-largest possible index.  Consider the value of $s$,
say $s^*$, for which $p$ is maximal and denote this particular $p$ by $p^*$.  Then any position to
the left of $a_{p^*}$ except for the beginning is eliminated as a possible site. Thus, the sites of
$\pi_{i-1}$ other than the beginning must comprise a set of consecutive positions, with this set
nonempty if $p^*\leq m-1$.

Let $j$ and $r$ be as before, but now in conjunction with $\{213,1322\}$, and we consider the same
cases.  If $j=r=1$, then similar reasoning yields the same result as before in this case.  If $j=1$
and $r>1$, then each position of offspring $\rho$ beyond the second letter is eliminated as a site
due to 213, with the first two positions active.  This gives $\binom{b_i-2+\ell-r}{b_i-2}$ possible
$\rho$ for every $r$, each with $\text{act}(\rho)=2$.  For the last two cases, assume for now $b_i
\geq 3$.  If $j=r>1$, then $\text{act}(\rho)=3$ if offspring $\rho$ possesses a single run of $i$,
since in this case the positions directly following the final two $i$'s are active in addition to the
initial site.  If the $i$'s occur as two runs in $\rho$, with the second run of length one, of which
there are $\ell-r$ possible $\rho$, then $\text{act}(\rho)=2$ as the only non-initial site directly
follows the penultimate added letter $i$.  For all other $\rho$
($\binom{b_i-2+\ell-r}{b_i-2}-(\ell-r+1)$ possibilities), we have $\text{act}(\rho)=1$ since all
positions beyond the first letter prior to the second-to-last added $i$ are eliminated as sites (due
to $1322$) as are all positions to the right of, but not directly following, the first run of $i$
(due to $213$).  If $1 < j <r \leq \ell$, we similarly get $\binom{b_i-2+\ell-r}{b_i-2}$ possible
$\rho$ each with $\text{act}(\rho)=1$ since $b_i\geq 3$.

Combining the previous cases yields the same succession rules with regard to avoiding the pattern
set $\{213,1322\}$ when $b_i \geq 3$ as those given above for $\{132,2213\}$.  Note that if $b_i=2$,
then one gets a single $\rho$ with $\text{act}(\rho)=3$ for each $r$ in the $j=r>1$ case and
$\text{act}(\rho)=2$ for each  $1<j<r$.  Thus when $b_i=2$ and $\ell \geq 2$, one gets $\ell
\rightsquigarrow 1~(0),~2~(\binom{\ell}{2}),~3~(\ell)$, with the indicated multiplicities, which is
in accord with the prior formula when $b_i=2$.  Finally, for all $b_i \geq 2$, we have
$1\rightsquigarrow b_i+1$.  Comparing the various successions rules above implies the desired result.
\end{proof}

\subsection{Equivalences by variations of the Simion-Schmidt correspondence.}

\begin{theorem}\label{123,1222}
 We have $\{123,1322\}\csim\{132,1223\}$.
\end{theorem}
\begin{proof}
Let $S=1^{b_1}\cdots k^{b_k}$. Consider a natural extension of the Simion-Schmidt correspondence
\cite{SS} which will establish $|\Av(S;132)|=|\Av(S;123)|$ that is defined as follows.  Let
$\ell_1>\cdots>\ell_r=1$ denote the set of left-right minima (lr min) of $\lambda \in \Av(S;132)$.
Then we may write $\lambda=\ell_1\lambda^{(1)}\cdots \ell_r\lambda^{(r)}$, where there is no $\ell_i$
in the section $\lambda^{(j)}$ for $1 \leq j <i$.  A letter belonging to some $\lambda^{(i)}$ that is
not equal to $\ell_i$ will be described as \emph{red}, with all other letters of $\lambda$ being
\emph{blue}.  Note that  one or more copies of an lr min $\ell_j$ is red if $\ell_j \in
\lambda^{(i)}$ for some $i>j$, with all other copies of $\ell_j$ blue (including the leftmost).
Also, $\lambda$ avoiding $132$ implies that the red letters occurring in each $\lambda^{(i)}$ are
increasing.  Let $\lambda'$ be obtained from $\lambda$ by considering the subsequence $S$ of
$\lambda$ comprising all of its red letters and rewriting the entries of $S$ in decreasing order,
leaving the blue letters unchanged in their positions.  One may verify that the mapping $\lambda
\mapsto \lambda'$ is a bijection.  Note that if an lr min $x$ of $\lambda$ occurs as a blue letter
with multiplicity $p \geq 1$ and as a red letter with multiplicity $s \geq 0$, then $x$ must occur as
blue and red letters with the same multiplicities in $\lambda'$.

To show $\{123,1322\}\csim\{132,1223\}$, we first make the following further definition.  Given $1 <
i \leq r$, let $\ell_i^*$ denote the largest member of $[\ell_i+1,\ell_{i-1}]$ occurring to the right
of the leftmost occurrence of $\ell_i$, with $\ell_1^*=k$ if $\ell_1<k$.  Note that $\ell_i^*$ need
not exist if $\ell_{i-1}=\ell_{i}+1$.  If $\pi \in \Av(S;132)$, consider for each $\ell_i$ the set
of elements in $[\ell_i+1,k]$ for which there is at least one letter occurring to the right of the
leftmost $\ell_i$ in $\pi$.  One may verify that $\pi$ avoiding the pattern $1223$ is equivalent to
the condition that the only red letters of $\pi$ that may occur more than once are those equal to
$\ell_j^*$ for some $j \in [r]$.  Note that the subset of $[k]$ comprising those members that
correspond to $\ell_j^*$ for some $j$ is the same for $\pi$ as it is for $\pi'$.  To see this,
consider cases on whether $\ell_i^*=\ell_{i-1}$ or $\ell_i^*<\ell_{i-1}$ for each $i>1$, with the
letter $k$ always possibly repeated as a red letter assuming $k>1$.  Since the multiplicity of each
red and blue letter is preserved by the mapping $\pi \mapsto \pi'$, the latter condition above is
equivalent to $\pi'$ avoiding $1322$, which implies the result.
\end{proof}

\noindent\emph{Remark:}
The proof of Theorem \ref{123,1222} may be generalized to show
$\{123,132^r\}\csim\{132,12^r3\}$ for all $r \geq 1$.

\begin{theorem}\label{121,1233}
 We have
$\{121,1233\}\csim\{121,2133\}$ and $\{112,1233\}\csim\{112,2133\}$.
\end{theorem}
\begin{proof}
We first define a bijection between $\Av(S;121,1233)$ and $\Av(S;121,2133)$, where
$S=1^{b_1}\cdots k^{b_k}$. Let $\pi \in \Av(S;121,1233)$.  We represent $\pi=\pi_1\pi_2\cdots
\pi_m$ as a $k$-ary word, where $m=b_1+\cdots+b_k$.  Let $a_1$ denote the largest $i \in [k]$ for
which $b_i>1$, assuming it exists, and let $i_1$ be the index $j \in [m]$ such that $\pi_j$
corresponds to the second rightmost occurrence of the letter $a_1$ in $\pi$.  Consider the
decomposition of the initial section
$$S_1=\pi_1\pi_2\cdots \pi_{i_1-1}=\alpha_1\beta_1\cdots \alpha_r\beta_r,$$
where $r \geq 1$ and $\alpha_i$ and $\beta_j$ denote runs of letters in $[a_1-1]$ and $[a_1,k]$,
respectively, with all $\alpha_i$ and $\beta_j$ nonempty except for possibly $\alpha_1$ and
$\beta_r$.  Note that all letters in $S_1$ belonging to $[a_1-1]$ are (weakly) decreasing since $\pi$
avoids $1233$.  We replace $S_1$ with $S_1'$ within $\pi_1=\pi$ to obtain $\pi_2$, where
$$S_1'=\begin{cases}
{\displaystyle
\text{rev}(\alpha_r)\beta_1\text{rev}(\alpha_{r-1})\beta_2\cdots\text{rev}(\alpha_1)\beta_r},
&\emph{if}~\alpha_1\neq\varnothing;\\
{\displaystyle
\beta_1\text{rev}(\alpha_r)\beta_2\text{rev}(\alpha_{r-1})\cdots\beta_{r-1}\text{rev}
(\alpha_2)\beta_r}, &\emph{if}~\alpha_1=\varnothing.\\
\end{cases}$$
Observe that $\pi_2$ contains no $2133$ in which the role of $3$ is played by $a_1$ since the
entries in $\cup_{i=1}^r\alpha_i$ are decreasing.  Note further that $\pi_1$ avoiding $121$ implies
that the $\alpha_i$ are pairwise disjoint and thus $\pi_2$ avoids $121$ as well.

We now consider the largest letter in $\pi_2$ that can play the role of a $3$  in a possible
occurrence of 2133.  Let $\pi_2=\pi_1^{(2)}\pi_2^{(2)}\cdots\pi_m^{(2)}$, as a word.  Let $a_2$, if
it exists, be the largest element of $[a_1-1]$ that is repeated whose second rightmost occurrence
corresponds to $\pi_\ell^{(2)}$ for some $\ell>i_1$.  We will denote this index $\ell$ by $i_2$. Note
that $\pi_2$ avoiding $121$ implies all letters $a_2$ occur to the right of the second rightmost
$a_1$.  Consider the decomposition of the initial section
$$S_2=\pi_1^{(2)}\pi_2^{(2)}\cdots \pi_{i_2-1}^{(2)}=\alpha_1^{2}\beta_1^{(2)}\cdots
\alpha_s^{(2)}\beta_s^{(2)},$$
where $s \geq 1$ and $\alpha_i^{(2)}$ and $\beta_j^{(2)}$ denote runs of letters in $[a_2-1]$ and
$[a_2,k]$, respectively, with only $\alpha_1^{(2)}$ and $\beta_s^{(2)}$ possibly empty.  Suppose that
the second rightmost $a_1$ occurs in $\beta_t^{(2)}$ for some $t \in [s]$.  Then the entries in
$\cup_{i=1}^t\alpha_i^{(2)}$ are (weakly) increasing via the first step of the algorithm above,
whereas those in $\cup_{i=t+1}^s\alpha_i^{(2)}$ are decreasing, with either set of entries possibly
empty (which always is the case if $s=t$ or if $s=1$, with $\alpha_1^{(2)}$ empty).  Note further
that $\pi_1$ avoiding $1233$ implies that the maximum of the set $\cup_{i=t+1}^s\alpha_i^{(2)}$ is
less than or equal the minimum of  $\cup_{i=1}^t\alpha_i^{(2)}$ assuming both sets are nonempty and
indeed this inequality is strict since $\pi_2$ avoids $121$.

Define $S_2'$ by
$$S_2'=\begin{cases}
{\displaystyle
\text{rev}(\alpha_s^{(2)})\beta_1^{(2)}\cdots\text{rev}(\alpha_{t+1}^{(2)})\beta_{s-t}^{(2)}
\alpha_1^{(2)}\beta_{s-t+1}^{(2)}\cdots \alpha_t^{(2)}\beta_s^{(2)}},
&\emph{if}~t<s~\emph{and}~\alpha_1^{(2)}\neq\varnothing;\\
{\displaystyle
\beta_1^{(2)}\text{rev}(\alpha_s^{(2)})\cdots\beta_{s-t}^{(2)}\text{rev}(\alpha_{t+1}^{(2)})
\beta_{s-t+1}^{(2)}\alpha_2^{(2)}\cdots\beta_{s-1}^{(2)}\alpha_t^{(2)}\beta_s^{(2)}},
&\emph{if}~t<s~\emph{and}~\alpha_1^{(2)}=\varnothing;\\
{\displaystyle S_2}, &\emph{if}~t=s.
\end{cases}$$
We replace $S_2$ by $S_2'$ in $\pi_2$, denoting the resulting multipermutation by $\pi_3$.  Note
that $\pi_3$ contains no $2133$ in which the role of 3 is played by a member of $[a_2,k]$, by the
maximality of $a_1$ and $a_2$.  Furthermore, since complete runs of letters belonging to $[a_2-1]$
are shifted in the transition from $\pi_2$ to $\pi_3$, no occurrence of $121$ is introduced.

Repeat the procedure above for $\pi_3=\pi_1^{(3)}\pi_2^{(3)}\cdots\pi_m^{(3)}$, considering for the
largest member $a_3$ of $[a_2-1]$ for which there is a letter occurring at least twice to the right
of $\pi_{i_2}^{(3)}$.  Continue the procedure until one produces $\pi_r$ for which there is no letter
in $[a_{r-1}-1]$ occurring to the right of the $(i_{r-1})$-st entry that is repeated.  Note that this
process must terminate since $a_{i+1}<a_i$ for all $i$.  One may verify $\pi_r \in \Av(S;121,2133)$.

Define $f\colon\Av(S;121,1233)\rightarrow \Av(S;121,2133)$ by setting $f(\pi)=\pi_r$ if $r>1$,
with $f(\pi)=\pi$ if $\pi$ contains no repeated letters.

To show that $f$ is reversible, first let $a_1'$ denote the largest repeated letter of $f(\pi)$.
Then let $a_i'$ be obtained from $a_{i-1}'$ for $i>1$ by considering the largest member of
$[a_{i-1}'-1]$ for which there is a repeated letter occurring to the right of the second rightmost
$a_{i-1}'$ in $f(\pi)$.  One can show that $a_i=a_i'$ for $1 \leq i <r$, as the relative order of the
penultimate occurrences of the $a_i$ is unchanged in each step of the process in the transformation
from $\pi_1$ to $\pi_r$ (though the exact positions of these occurrences might change).  Thus, $f$
may be reversed as follows.  First undo the last step of the above procedure starting with $\pi_r$
upon considering the position of the penultimate $a_{r-1}$.  Note that $a_{r-1}$ is actually the
rightmost letter occurring at least twice in $\pi_r$.  Then undo the second-to-last step considering
$a_{r-2}$ in $\pi_{r-1}$ and so on, successively, until one obtains $\pi_2$.  Inverting the mapping
described in the first paragraph above then recovers $\pi_1=\pi$ from $\pi_2$.

The bijection $f$ also applies to $\{112,1233\}\csim\{112,2133\}$.  Note that there is no $112$
introduced in going from $\pi_{i-1}$ to $\pi_i$ for all $i$ since the letters to the left of
$a_{i-1}$ belonging to $[a_{i-1}-1]$ are seen to be distinct, upon proceeding inductively (the $i=1$
case, by assumption).  Also, one may obtain $\pi_i$ from $\pi_{i-1}$ in the proof of this equivalence
simply by reversing the order of the terms in the subsequence of $\pi_{i-1}$ comprising all members
of $[a_{i-1}-1]$ to the left of the penultimate $a_{i-1}$ letter since we need not avoid $121$ in
this case.
\end{proof}

\section{Concluding remarks}

In Table \ref{tabnontrivialw} below is a list of the members of the non-singleton $(3,4)$ Wilf-equivalence classes for compositions, up to symmetry (i.e., reversal), along with their respective theorem references.  Note that $m$-equivalence of two sets of patterns clearly implies the equivalence of the sets with respect to avoidance in compositions. \pagebreak

\begin{table}[htp]
\begin{tabular}{|p{4.15cm}|p{2.6cm}||p{4.15cm}|p{3.6cm}|} \hline
  $\{\sigma,\tau\}$&Reference&$\{\sigma,\tau\}$&Reference\\\hline\hline
$\{111,1221\}$, $\{111,2112\}$&Thm. \ref{teo1}&
$\{111,1123\}$, $\{111,1132\}$&(c) Prop. \ref{pro-ferrers}\\\hline
$\{111,1223\}$, $\{111,1232\}$, $\{111,1322\}$, $\{111,2123\}$, $\{111,2132\}$, $\{111,2213\}$&Prop. \ref{pro-ferrers}&
$\{111,1233\}$, $\{111,2133\}$&Prop. \ref{pro-ferrers}\\\hline
$\{111,1234\}$, $\{111,1243\}$, $\{111,1432\}$, $\{111,2134\}$, $\{111,2143\}$, $\{111,3214\}$&Prop. \ref{pro-ferrers}&
$\{112,1111\}$, $\{121,1111\}$&(c) Prop. \ref{pro-ferrers}\\\hline
$\{112,1211\}$, $\{121,1112\}$&Thm. \ref{teo1}&
$\{112,2121\}$, $\{121,1122\}$&Thm. \ref{112,2121}\\\hline
$\{112,1231\}$, $\{121,1132\}$&Thm. \ref{teormulti1}&
$\{112,2122\}$, $\{121,1222\}$&Thm. \ref{112,2122}\\\hline
$\{112,2212\}$, $\{121,2122\}$&Thm. \ref{112,2212}&
$\{112,2312\}$, $\{121,1223\}$, $\{121,2213\}$&Thm. \ref{112,2312}\\\hline
$\{112,1232\}$, $\{112,2132\}$, $\{121,1322\}$&Thm. \ref{121,1322}&
$\{121,1233\}$, $\{121,2133\}$&Thm. \ref{121,1233} \\\hline
$\{112,1233\}$, $\{112,2133\}$&Thm. \ref{121,1233} &
$\{121,1234\}$, $\{121,2134\}$&Thm. \ref{121,1243}\\\hline
$\{121,1243\}$, $\{121,2143\}$&Thm. \ref{121,1243}&
$\{112,1243\}$, $\{112,2143\}$&Thm. \ref{121,1243}\\\hline
$\{121,1342\}$, $\{121,3142\}$&Conj. \ref{121,1342}&
$\{112,2314\}$, $\{112,3124\}$&Prop. \ref{pro-ferrers}\\\hline
$\{112,1234\}$, $\{112,2134\}$, $\{112,3214\}$&Prop. \ref{pro-ferrers}&
$\{122,1111\}$, $\{212,1111\}$&Prop. \ref{pro-ferrers}\\\hline
$\{123,1111\}$, $\{132,1111\}$, $\{213,1111\}$&Prop. \ref{pro-ferrers}&
$\{122,1121\}$, $\{212,1112\}$&(c) Thm. \ref{112,2122}\\\hline
$\{122,1211\}$, $\{212,1121\}$&(c) Thm. \ref{112,2212}&
$\{123,1211\}$, $\{213,1211\}$&Thm. \ref{123,1112}\\\hline
$\{123,1121\}$, $\{132,1211\}$, $\{213,1121\}$&Thm. \ref{1231121th}&
$\{123,1112\}$, $\{132,1121\}$, $\{213,1112\}$&Prop. \ref{pro-ferrers}, Thm. \ref{123,1112}\\\hline
$\{122,2121\}$, $\{212,1122\}$&(c) Thm. \ref{112,2121}&
$\{122,2212\}$, $\{212,1222\}$&(c) Thm. \ref{teo1}\\\hline
$\{212,1123\}$, $\{212,1132\}$&(c) Thm. \ref{121,1233}&
$\{122,1123\}$, $\{122,1132\}$&(c) Thm. \ref{121,1233}\\\hline
$\{122,2312\}$, $\{212,1223\}$, $\{212,1322\}$&(c) Thm. \ref{112,2312}&
$\{122,2123\}$, $\{122,2132\}$, $\{212,2213\}$&(c) Thm. \ref{121,1322}\\\hline
$\{122,3123\}$, $\{212,2133\}$&(c) Thm. \ref{teormulti1}&
$\{122,2134\}$, $\{122,2143\}$&(c) Thm. \ref{121,1243}\\\hline
$\{212,1234\}$, $\{212,1243\}$&(c) Thm. \ref{121,1243}&
$\{212,3124\}$, $\{212,3142\}$&Conj. \ref{121,1342}\\\hline
$\{212,2134\}$, $\{212,2143\}$&(c) Thm. \ref{121,1243}&
$\{122,1342\}$, $\{122,1423\}$&(c) Prop. \ref{pro-ferrers}\\\hline
$\{122,1234\}$, $\{122,1243\}$, $\{122,1432\}$&(c) Prop. \ref{pro-ferrers}&
$\{123,2212\}$, $\{132,2212\}$&(c) Thm. \ref{123,1112}\\\hline
$\{123,2122\}$, $\{132,2122\}$, $\{213,2212\}$&(c) Thm. \ref{1231121th}&
$\{123,1222\}$, $\{132,1222\}$, $\{213,2122\}$&(c) Prop. \ref{pro-ferrers}, Thm. \ref{123,1112}\\\hline
$\{132,2213\}$, $\{213,1322\}$&Thm. \ref{132,2213}&
$\{123,1322\}$, $\{132,1223\}$&Thm. \ref{123,1222}\\\hline
$\{123,2213\}$, $\{213,1223\}$&(c) Thm. \ref{123,1222}&&\\\hline
\end{tabular}
\caption{Non-trivial Wilf classes for compositions, where (c) stands for complement.}\label{tabnontrivialw}
\end{table}

Extending the arguments above (at times, treating separately the $r=1$ case) yields the following
set of generalized equivalences, some of which have already been mentioned in prior remarks.

\begin{theorem}\label{general}
 The following equivalences hold, where it is assumed $r \geq 1$ unless stated otherwise:
\begin{align*}
(i)&~\{1^{r+1},\tau\}\csim\{1^{r+1},\tau^c\},\\
(ii)&~\{112,121^{r-1}\}\csim\{121,1^r2\},\\
(iii)&~\{112,1231^{r-1}\}\csim\{121,1^r32\},\\
(iv)&~\{112,1232^{r-1}\}\csim\{112,2132^{r-1}\}\csim\{121,132^r\},\\
(v)&~\{112,123^r\}\csim\{112,213^r\},\\
(vi)&~\{112,212^r\}\csim\{121,12^{r+1}\},\\
(vii)&~\{112,2^r12\}\csim\{121,212^r\}, \qquad r \geq 2,\\
(viii)&~\{112,2312^{r-1}\}\csim\{121,12^r3\}\csim\{121,2^r13\}, \qquad r \geq 2,\\
(ix)&~\{121,123^r\}\csim\{121,213^r\},\\
(x)&~\{123,1^r2\}\csim\{132,1^{r-1}21\},\\
(xi)&~\{123,1^r21\}\csim\{213,1^r21\},\\
(xii)&~\{123,121^r\}\csim\{213,121^r\},\\
(xiii)&~\{123,132^r\}\csim\{132,12^r3\},\\
(xiv)&~\{132,2^r13\}\csim\{213,132^r\},
\end{align*}
where $\tau$ in (i) denotes a permutation of the multiset $1^r\cdots k^r$ for some $k$.
\end{theorem}

Taking the appropriate value of $r$ in each part of Theorem \ref{general} above yields the $(3,4)$ equivalence which it generalizes.

Combining the $r=1$ cases above in Theorem \ref{general}, together with the strong equivalence of $123$, $132$ and $213$ and also of $112$ and $121$,
yields the following complete list of multiset equivalences for $(3,3)$ up to symmetry:
\begin{itemize}
\item $\{111,123\}\csim\{111,132\}\csim\{111,213\}$,
\item $\{111,112\}\csim\{111,121\}$,
\item $\{112,123\}\csim\{112,213\}\csim\{121,132\}$,
\item $\{112,212\}\csim\{121,122\}$,
\item $\{121,123\}\csim\{121,213\}$. \medskip
\end{itemize}

Let $C_n(\tau,\tau')$ denote the set of all compositions of $n$ that avoid the patterns $\tau$ and $\tau'$ and let $c_n(\tau,\tau')=|C_n(\tau,\tau')|$.  Our enumeration data suggest two possible compositional equivalences that we are unable to prove by our present
methods, which we state here as a conjecture.

\begin{conjecture}\label{121,1342}
If $n \geq 1$, then
$$c_n(121,1342)=c_n(121,3142) \quad\text{and}\quad c_n(212,3124)=c_n(212,3142).$$
\end{conjecture}

Note that numerically we have confirmed both equivalences in Conjecture \ref{121,1342} for all $n$ up to $n=35$.  However, the number of permutations of the multiset $1^2\cdots5^2$ that avoid $\{121,1342\}$ is seen to differ from the number of permutations that avoid $\{121,3142\}$  (46566 vs. 45969).  Taking reverse complements gives the same story for $\{212,3124\}$ and $\{212,3142\}$.  Thus, if demonstrated, the equivalences in Conjecture \ref{121,1342} would provide two examples of pattern sets that are Wilf-equivalent with respect to compositions, but are not multiset equivalent.

\section{Appendix 1}

Below we present, up to symmetry, the cardinalities of all $C_n(\tau,\tau')$ when $n=24$, where $\tau$ denotes a $3$-letter and $\tau'$ a $4$-letter pattern. Note that in a few cases, some values for $n>24$ are given which are required to differentiate a specific class from another.

{
\begin{multicols}{2}
\begin{itemize}
\item $\{111,1212\}$: 104335,
\item[$\blacksquare$] $\{111,1221\}$, $\{111,2112\}$: 104873,
\item $\{111,1122\}$: 105957,
\item $\{111,1312\}$: 84928,
\item $\{111,1231\}$: 86035,
\item $\{111,2113\}$: 86215,
\item $\{111,1213\}$: 86218,
\item[$\blacksquare$] $\{111,1123\}$, $\{111,1132\}$: 86641,
\item[$\blacksquare$] {$\{111,1223\}$, $\{111,1232\}$, $\{111,1322\}$, $\{111,2123\}$, $\{111,2132\}$, $\{111,2213\}$}: 95697,
\item $\{111,2313\}$: 109460,
\item $\{111,1332\}$: 109787,
\item $\{111,1323\}$: 109789,
\item $\{111,3123\}$: 109781,
\item[$\blacksquare$] {$\{111,1233\}$, $\{111,2133\}$}: 109932,
\item $\{111,2413\}$: 94579,146512,
\item $\{111,1423\}$: 94579,146775,
\item $\{111,1342\}$: 94690,
\item $\{111,2314\}$: 94800,
\item $\{111,3124\}$: 94816,
\item[$\blacksquare$] {$\{111,1234\}$, $\{111,1243\}$, $\{111,1432\}$, $\{111,2134\}$, $\{111,2143\}$, $\{111,3214\}$}: 94939,
\item $\{111,1324\}$: 94956,
\item[$\blacksquare$] {$\{112,1111\}$, $\{121,1111\}$}: 52825,
\item $\{112,2111\}$: 37484,
\item[$\blacksquare$] {$\{112,1211\}$, $\{121,1112\}$}: 43592,
\item $\{112,2211\}$: 46008,
\item $\{112,1221\}$: 56883,
\item[$\blacksquare$] {$\{112,2121\}$, $\{121,1122\}$}: 58994,
\item $\{121,2112\}$: 66160,
\item $\{112,3211\}$: 28809,
\item $\{112,2311\}$: 33701,
\item $\{112,2221\}$: 62378,
\item $\{112,3121\}$: 45347,
\item $\{112,1321\}$: 49080,
\item $\{121,2113\}$: 51680,
\item $\{121,1123\}$: 51887,
\item $\{112,2131\}$: 57398,
\item[$\blacksquare$] {$\{112,1231\}$, $\{121,1132\}$}: 58001,
\item $\{112,1222\}$: 71059,
\item[$\blacksquare$] {$\{112,2122\}$, $\{121,1222\}$}: 72586,
\item[$\blacksquare$] {$\{112,2212\}$, $\{121,2122\}$}: 77327,
\item $\{112,3221\}$: 45504,
\item $\{112,2321\}$: 51758,
\item $\{112,3122\}$: 60722,
\item $\{112,3212\}$: 63624,
\item $\{112,1322\}$: 63020,
\item[$\blacksquare$] {$\{112,2312\}$, $\{121,1223\}$ ,$\{121,2213\}$}: 66262,
\item $\{121,2123\}$: 71678,
\item[$\blacksquare$] {$\{112,1232\}$, $\{112,2132\}$, $\{121,1322\}$}: 70572,
\item $\{112,3321\}$: 61826,
\item $\{112,2331\}$: 67910,
\item $\{112,3231\}$: 68588,
\item $\{112,3312\}$: 75161,
\item $\{112,1332\}$: 74353,
\item $\{112,3132\}$: 74672,
\item[$\blacksquare$] {$\{121,1233\}$, $\{121,2133\}$}: 77444,
\item $\{121,3123\}$: 80111,
\item[$\blacksquare$] {$\{112,1233\}$, $\{112,2133\}$}: 79160,
\item $\{112,2313\}$: 79443,
\item $\{121,1332\}$: 79444,
\item $\{112,1323\}$: 79510,
\item $\{112,3123\}$: 79638,
\item $\{112,3213\}$: 79713,
\item $\{121,1323\}$: 80707,118366,172460,\break250527,362373,521428,
\item $\{121,2313\}$: 80707,118366,172460,\break250527,362373,521429,
\item $\{112,4321\}$: 45291,
\item $\{112,3421\}$: 50435,
\item $\{112,4231\}$: 57822,
\item $\{112,2431\}$: 59077,
\item $\{112,3241\}$: 65123,
\item $\{112,2341\}$: 65182,
\item $\{112,4312\}$: 65130,
\item $\{112,4132\}$: 67293,
\item $\{112,1432\}$: 68218,
\item $\{112,3412\}$: 67262,
\item $\{112,3142\}$: 72021,
\item $\{112,1342\}$: 72102,
\item $\{121,3124\}$: 74030,
\item $\{121,3214\}$: 74040,
\item[$\blacksquare$] {$\{121,1234\}$, $\{121,2134\}$}: 74041,
\item[$\blacksquare$] {$\{121,1243\}$, $\{121,2143\}$}: 76542,
\item $\{112,4213\}$: 75266,
\item $\{112,4123\}$: 75303,
\item $\{112,2413\}$: 75621,
\item $\{112,1423\}$: 75647,
\item[$\blacksquare$] {$\{112,1243\}$, $\{112,2143\}$}: 76193,
\item $\{121,2314\}$: 76770,
\item $\{121,1324\}$: 76780,
\item[$\blacksquare$] $\{121,1342\}$, $\{121,3142\}$: 78034,113633,\break164728,
\item $\{121,1423\}$: 78034,113633,164727,
\item $\{121,1432\}$: 78044,
\item[$\blacksquare$] $\{112,2314\}$, $\{112,3124\}$: 79157,
\item[$\blacksquare$] $\{112,1234\}$, $\{112,2134\}$, $\{112,3214\}$: 79198,115565,167895,
\item $\{112,1324\}$: 79198,115565,167896,
\item[$\blacksquare$] {$\{122,1111\}$, $\{212,1111\}$}: 138512,
\item[$\blacksquare$] {$\{123,1111\}$, $\{132,1111\}$, $\{213,1111\}$}: 83829,
\item $\{122,2111\}$: 97469,
\item $\{122,1112\}$: 145591,
\item[$\blacksquare$] {$\{122,1121\}$, $\{212,1112\}$}: 149406,
\item[$\blacksquare$] {$\{122,1211\}$, $\{212,1121\}$}: 162336,
\item $\{123,2111\}$: 45027,
\item $\{132,2111\}$: 55190,
\item $\{213,2111\}$: 90128,
\item $\{132,1112\}$: 101522,
\item[$\blacksquare$] {$\{123,1211\}$, $\{213,1211\}$}: 108281,
\item[$\blacksquare$] {$\{123,1121\}$, $\{132,1211\}$, $\{213,1121\}$}: 120344,
\item[$\blacksquare$] {$\{123,1112\}$, $\{132,1121\}$, $\{213,1112\}$}: 133919,
\item $\{122,2211\}$: 396868,
\item $\{122,2112\}$: 444097,
\item[$\blacksquare$] {$\{122,2121\}$, $\{212,1122\}$}: 445294,
\item $\{212,1221\}$: 448694,
\item $\{122,2221\}$: 593122,
\item[$\blacksquare$] $\{122,2212\}$, $\{212,1222\}$: 607805,
\item $\{122,2311\}$: 207967,
\item $\{122,3211\}$: 217656,
\item $\{122,2131\}$: 217373,
\item $\{122,2113\}$: 218288,
\item $\{122,3112\}$: 249010,
\item $\{122,3121\}$: 251890,
\item[$\blacksquare$] {$\{212,1123\}$, $\{212,1132\}$}: 286877,
\item $\{212,1231\}$: 306089,
\item $\{212,2113\}$: 318167,
\item $\{122,1312\}$: 318616,
\item $\{122,1231\}$: 322179,
\item[$\blacksquare$] {$\{122,1123\}$, $\{122,1132\}$}: 321098,
\item $\{122,1213\}$: 326103,
\item $\{122,1321\}$: 331513,
\item $\{212,1213\}$: 326937,505465,776447,
\item $\{212,1312\}$: 326937,505465,776446,
\item $\{122,2231\}$: 472765,
\item $\{122,2213\}$: 481694,
\item $\{122,2321\}$: 506314,
\item[$\blacksquare$] {$\{122,2312\}$, $\{212,1223\}$, $\{212,1322\}$}: 518302,
\item[$\blacksquare$] {$\{122,2123\}$, $\{122,2132\}$, $\{212,2213\}$}: 542447,
\item $\{122,3221\}$: 569643,
\item $\{122,3212\}$: 586797,
\item $\{212,1232\}$: 628352,
\item $\{122,3321\}$: 639459,
\item $\{122,3231\}$: 636726,
\item $\{122,3213\}$: 642427,
\item $\{122,3312\}$: 645365,
\item $\{212,1332\}$: 658101,
\item $\{212,1233\}$: 658147,
\item $\{122,3132\}$: 668842,
\item[$\blacksquare$] {$\{122,3123\}$, $\{212,2133\}$}: 669649,
\item $\{122,3241\}$: 502005,
\item $\{122,2431\}$: 511743,
\item $\{122,2341\}$: 512408,
\item $\{122,3214\}$: 510131,
\item $\{122,2413\}$: 516402,
\item $\{122,2314\}$: 516912,
\item[$\blacksquare$] $\{122,2134\}$, $\{122,2143\}$: 529470,
\item $\{122,3421\}$: 524533,
\item $\{122,4231\}$: 525914,
\item $\{122,4321\}$: 531152,
\item $\{122,4213\}$: 531833,
\item $\{122,3412\}$: 531977,
\item $\{122,4312\}$: 539208,
\item $\{212,1423\}$: 562177,
\item $\{212,1324\}$: 562233,
\item $\{212,1342\}$: 565100,
\item[$\blacksquare$] $\{212,1234\}$, $\{212,1243\}$: 565155,904528,
\item $\{212,1432\}$: 565155,904529,
\item $\{122,3142\}$: 550948,
\item $\{122,3124\}$: 552176,
\item $\{122,4132\}$: 565675,
\item $\{122,4123\}$: 565838,
\item $\{212,2314\}$: 583686,
\item[$\blacksquare$] $\{212,2413\}$, $\{212,3124\}$: 583691,
\item $\{212,3214\}$: 583743,
\item[$\blacksquare$] $\{212,2134\}$, $\{212,2143\}$: 595438,
\item[$\blacksquare$] $\{122,1342\}$, $\{122,1423\}$: 601322,
\item[$\blacksquare$] $\{122,1234\}$, $\{122,1243\}$, $\{122,1432\}$: 602740,
\item $\{122,1324\}$: 602746,
\item $\{132,2211\}$: 157727,
\item $\{213,1122\}$: 203230,
\item $\{123,2211\}$: 186726,
\item $\{213,1221\}$: 214850,
\item $\{132,2121\}$: 228554,
\item $\{123,2112\}$: 262501,
\item $\{213,2211\}$: 224583,
\item $\{123,2121\}$: 261305,
\item $\{132,2112\}$: 275288,
\item $\{123,1122\}$: 273983,
\item $\{213,1212\}$: 279782,
\item $\{132,1221\}$: 288149,
\item $\{132,1122\}$: 312538,
\item $\{123,1221\}$: 302929,
\item $\{213,2121\}$: 305159,
\item $\{132,1212\}$: 338776,
\item $\{213,2112\}$: 342635,
\item $\{123,1212\}$: 351544,
\item $\{123,3211\}$: 195977,
\item $\{213,3112\}$: 242902,
\item $\{123,3112\}$: 212143,
\item $\{213,3121\}$: 246356,
\item $\{123,3121\}$: 238941,
\item $\{213,3211\}$: 234061,
\item $\{132,2311\}$: 449163,
\item $\{132,2113\}$: 354774,
\item $\{213,1312\}$: 391242,
\item $\{213,1321\}$: 380778,
\item $\{132,1213\}$: 399970,
\item $\{123,1312\}$: 410195,
\item $\{132,1123\}$: 414782,
\item $\{123,1321\}$: 415510,
\item $\{132,2131\}$: 494818,
\item $\{213,1132\}$: 503790,
\item $\{123,1132\}$: 529612,
\item $\{132,1231\}$: 533009,
\item $\{132,2221\}$: 339146,
\item $\{213,1222\}$: 358327,
\item $\{123,2221\}$: 378373,
\item $\{213,2221\}$: 390069,
\item[$\blacksquare$] {$\{123,2212\}$, $\{132,2212\}$}: 418354,
\item[$\blacksquare$] {$\{123,2122\}$, $\{132,2122\}$, $\{213,2212\}$}: 440827,
\item[$\blacksquare$] {$\{123,1222\}$, $\{132,1222\}$, $\{213,2122\}$}: 485327,
\item $\{132,3211\}$: 521377,
\item $\{123,2311\}$: 524744,
\item $\{132,3121\}$: 570656,
\item $\{213,2311\}$: 599187,
\item $\{132,3112\}$: 601383,
\item $\{123,2131\}$: 641291,
\item $\{213,1231\}$: 671515,
\item $\{123,2113\}$: 744984,
\item $\{213,1123\}$: 746159,
\item $\{213,3122\}$: 342362,
\item $\{123,3221\}$: 343971,
\item $\{213,3221\}$: 340675,
\item $\{123,3212\}$: 382554,
\item $\{123,3122\}$: 381825,
\item $\{213,3212\}$: 394852,
\item $\{132,2231\}$: 526791,
\item[$\blacksquare$] {$\{132,2213\}$, $\{213,1322\}$}: 458220,
\item $\{132,2123\}$: 506438,
\item[$\blacksquare$] {$\{123,1322\}$, $\{132,1223\}$}: 533568,
\item $\{132,3221\}$: 603342,
\item $\{123,2321\}$: 686218,
\item $\{132,3212\}$: 653346,
\item $\{123,2312\}$: 694424,
\item $\{213,2321\}$: 697009,
\item $\{132,3122\}$: 681004,
\item $\{213,2312\}$: 732527,
\item $\{213,1232\}$: 753461,
\item $\{123,2132\}$: 771116,
\item $\{123,2231\}$: 854373,
\item $\{132,2321\}$: 859462,
\item $\{213,2231\}$: 875649,
\item $\{132,2312\}$: 881784,
\item[$\blacksquare$] {$\{123,2213\}$, $\{213,1223\}$}: 902785,
\item $\{213,3312\}$: 708649,
\item $\{123,3321\}$: 718470,
\item $\{123,3312\}$: 714199,
\item $\{213,3321\}$: 713750,
\item $\{132,2331\}$: 815573,
\item $\{213,1332\}$: 759017,
\item $\{123,3132\}$: 801324,
\item $\{132,2313\}$: 850475,
\item $\{213,3132\}$: 835356,
\item $\{132,2133\}$: 836368,
\item $\{123,1332\}$: 844730,
\item $\{132,1233\}$: 852658,
\item $\{123,2331\}$: 881288,
\item $\{213,3231\}$: 901887,
\item $\{132,3231\}$: 896638,
\item $\{132,3213\}$: 903969,
\item $\{132,3123\}$: 915698,
\item $\{213,1323\}$: 920937,
\item $\{123,2313\}$: 930903,
\item $\{123,2133\}$: 935140,
\item $\{213,3123\}$: 954148,
\item $\{132,3321\}$: 940340,
\item $\{123,3231\}$: 954767,
\item $\{132,3312\}$: 955697,
\item $\{213,2331\}$: 959257,
\item $\{213,1233\}$: 972460,
\item $\{123,3213\}$: 979479,
\item $\{213,4312\}$: 694835,
\item $\{123,4321\}$: 705999,
\item $\{123,4312\}$: 702556,
\item $\{213,4321\}$: 700152,
\item $\{132,2341\}$: 850077,
\item $\{213,1432\}$: 787114,
\item $\{123,4132\}$: 779792,
\item $\{132,2314\}$: 833222,
\item $\{213,4132\}$: 795511,
\item $\{132,2134\}$: 834511,
\item $\{132,1234\}$: 851461,
\item $\{123,1432\}$: 857921,
\item $\{213,4231\}$: 873579,
\item $\{123,2431\}$: 882598,
\item $\{132,3241\}$: 922782,
\item $\{213,3412\}$: 902041,
\item $\{123,3421\}$: 901200,
\item $\{132,3214\}$: 897047,
\item $\{123,3412\}$: 900841,
\item $\{213,3421\}$: 906726,
\item $\{213,1423\}$: 918238,
\item $\{132,3124\}$: 915830,
\item $\{123,2413\}$: 922421,
\item $\{132,4231\}$: 930808,
\item $\{213,4123\}$: 922049,
\item $\{132,4213\}$: 936394,
\item $\{213,1342\}$: 944445,
\item $\{132,4123\}$: 946103,
\item $\{123,3142\}$: 958210,
\item $\{123,2143\}$: 958787,
\item $\{132,3421\}$: 960869,
\item $\{123,4231\}$: 943423,
\item $\{132,3412\}$: 970430,
\item $\{213,2431\}$: 982832,
\item $\{132,4321\}$: 965760,
\item $\{123,4213\}$: 967564,
\item $\{123,3241\}$: 980259,
\item $\{132,4312\}$: 978603,
\item $\{213,1243\}$: 1002198,
\item $\{213,2341\}$: 1004697,
\item $\{123,3214\}$: 1012078,
\item $\{213,1234\}$: 1019326.
\end{itemize}
\end{multicols}}

\section{Appendix 2}

Below we present, up to symmetry, the cardinalities of all $C_n(\tau,\tau')$ when $n=25$, where $\tau$ and $\tau'$ denote patterns of length three.

\begin{multicols}{2}
\begin{itemize}
\item $\{112,321\}$:	9421,
\item $\{112,231\}$:	17355,
\item $\{112,211\}$:	17481,
\item $\{112,121\}$:	26731,
\item $\{112,312\}$:	32058,
\item[$\blacksquare$] $\{111,123\}$, $\{111,132\}$, $\{111,213\}$:	34552,
\item $\{112,132\}$:	38507,
\item $\{112,221\}$:	40883,
\item[$\blacksquare$] $\{111,112\}$, $\{111,121\}$:	44452,
\item[$\blacksquare$] $\{121,123\}$, $\{121,213\}$:	45077,
\item[$\blacksquare$] $\{112,123\}$, $\{112,213\}$, $\{121,132\}$:	57375,
\item $\{112,122\}$:	57556,
\item[$\blacksquare$] $\{112,212\}$, $\{121,122\}$:	62204,
\item $\{121,212\}$:	78113,
\item[$\blacksquare$] $\{111,122\}$, $\{111,212\}$:	79333,
\item $\{122,231\}$:	128143,
\item $\{122,213\}$:	145646,
\item $\{122,321\}$:	147169,
\item $\{122,312\}$:	170326,
\item $\{123,321\}$:	190404,
\item $\{213,312\}$:	205612.
\item $\{123,312\}$:	206502,
\item[$\blacksquare$] $\{123,212\}$, $\{132,212\}$:	219024,
\item[$\blacksquare$] $\{122,123\}$, $\{122,132\}$, $\{212,213\}$:	272449,
\item $\{132,213\}$:	459083,
\item $\{132,231\}$:	498583,
\item $\{123,132\}$:	500733,
\item $\{122,221\}$:	529581,
\item $\{122,212\}$:	576244,
\item $\{123,231\}$:	801464,
\item $\{132,312\}$:	866953,
\item $\{123,213\}$:	1005036,
\end{itemize} 	
\end{multicols}

\end{document}